\pdfoutput=1
\documentclass[11pt]{article}
\usepackage{dsfont}

\usepackage[letterpaper,margin=1in]{geometry}
\newcommand\blfootnote[1]{%
\begingroup
  \renewcommand\thefootnote{}\footnote{#1}%
  \addtocounter{footnote}{-1}%
  \endgroup
  }

\usepackage[dvipsnames,table,xcdraw]{xcolor}
\usepackage[utf8]{inputenc}

\def\colorful{0}

\ifnum\colorful=1
\usepackage[normalem]{ulem}
\newcommand{\new}[1]{{\blue #1}}
  \newcommand{\apnote}[1]{\footnote{{\bf [Ankit: {#1}\bf ]}}}
  \newcommand{\aanote}[1]{\footnote{{\bf [Amir: {#1}\bf ]}}}
  \newcommand{\plnote}[1]{\footnote{{\bf [Po-Ling: {#1}\bf ]}}}
  \newcommand{\vjnote}[1]{\footnote{{\bf [Varun: {#1}\bf ]}}}
\newcommand{\todo}[1]{{\textbf{ [{\red{Todo}}: {#1}]}}}
\else
\newcommand{\new}[1]{{#1}}
\newcommand{\apnote}[1]{}
\newcommand{\aanote}[1]{}
\newcommand{\plnote}[1]{}
\newcommand{\vjnote}[1]{}
\newcommand{\todo}[1]{}
\fi

\usepackage[T1]{fontenc}    %
\usepackage{url}	    %
\usepackage{booktabs}	    %
\usepackage{nicefrac}	    %

\usepackage{amsthm,amsfonts,amsmath,amssymb,epsfig,color,float,graphicx,verbatim,
enumitem}

\usepackage{algpseudocode,algorithm,algorithmicx}  

\usepackage{bbm}
\usepackage{caption}

\usepackage[
backend=biber,
style=alphabetic,
maxbibnames=15,
maxalphanames=10,
minalphanames=6,
doi=false,
isbn=false,
url=false,
eprint=false
]{biblatex}
\usepackage[colorlinks,citecolor=blue,linkcolor=magenta,bookmarks=true]{hyperref}

\usepackage[nameinlink,capitalise]{cleveref}

\usepackage{thm-restate}

\def\E{\mathbb E}
\def\P{\mathbb P}
\def\R{\mathbb R}

\def\N{\mathbb N}

\newcommand\numberthis{\addtocounter{equation}{1}\tag{\theequation}}

\let\vec\mathbf

\newcommand{\bH}{\vec{H}}
\newcommand{\bI}{\vec{I}}

\newcommand{\bL}{\vec{L}}

\newcommand{\bT}{\vec{T}}

\newcommand{\poly}{\mathrm{poly}}

\newcommand{\cT}{\mathcal{T}}
\newcommand{\cE}{\mathcal{E}}

\newcommand{\cP}{\mathcal{P}}
\newcommand{\cY}{\mathcal{Y}}
\newcommand{\cI}{\mathcal{I}}

\newcommand{\cH}{\mathcal{H}}
\newcommand{\cJ}{\mathcal{J}}
\newcommand{\cR}{\mathcal{R}}
\newcommand{\cS}{\mathcal{S}}
\newcommand{\cA}{\mathcal{A}}
\newcommand{\cC}{\mathcal{C}}
\newcommand{\cX}{\mathcal{X}}

\newcommand{\cB}{\mathcal{B}}

\newcommand{\dtv}{d_{\mathrm{TV}}}
\newcommand{\hel}{d_{\mathrm{h}}}

\newcommand{\cTT}{\mathcal{T}^\textup{\texttt{thresh}}}

\newcommand{\RR}{ \bT_{\mathrm{RR} }}
\newcommand{\ext}{{\mathrm{ext}}}

\newcommand{\nstar}{\mathrm{n}^*}

\newcommand{\conv}{\operatorname{conv}}

\newcommand{\real}{\ensuremath{\mathbb{R}}}
\newcommand{\Ber}{\mathrm{Ber}}

\crefformat{equation}{(#2#1#3)}
\crefname{equation}{Equation}{Equations}
\crefname{lemma}{Lemma}{Lemmata}
\crefname{claim}{Claim}{Claims}
\crefname{fact}{Fact}{Facts}
\crefname{theorem}{Theorem}{Theorems}
\crefname{proposition}{Proposition}{Propositions}
\crefname{corollary}{Corollary}{Corollaries}
\crefname{remark}{Remark}{Remarks}
\crefname{definition}{Definition}{Definitions}
\crefname{question}{Question}{Questions}
\crefname{condition}{Condition}{Conditions}
\crefname{figure}{Figure}{Figures}

\newtheorem{theorem}{Theorem}[section]
\newtheorem{lemma}[theorem]{Lemma}

\newtheorem{claim}[theorem]{Claim}
\newtheorem{proposition}[theorem]{Proposition}
\newtheorem{corollary}[theorem]{Corollary}

\theoremstyle{definition}
\newtheorem{fact}[theorem]{Fact}
\newtheorem{definition}[theorem]{Definition}

\newtheorem{condition}[theorem]{Condition}

\newtheorem{remark}[theorem]{Remark}

\allowdisplaybreaks

\theoremstyle{definition}

\usepackage{color}
\definecolor{Red}{rgb}{1,0,0}
\definecolor{Blue}{rgb}{0,0,1}
\definecolor{DGreen}{rgb}{0,0.55,0}
\definecolor{Purple}{rgb}{.75,0,.25}

\def\red{\color{Red}}

\title{Simple Binary Hypothesis Testing under Local Differential Privacy  and Communication Constraints\blfootnote{This paper was presented in part at COLT 2023 as an extended abstract~\cite{PenEtal23colt}.}
}

\author{
\begin{tabular}{c c}
Ankit Pensia & Amir R.\ Asadi\\
IBM Research & University of Cambridge\\
{\tt ankitp@ibm.com} & {\tt aa2345@cam.ac.uk}
\\
\\
Varun Jog & Po-Ling Loh\\ 
University of Cambridge & University of Cambridge\\
{\tt vj270@cam.ac.uk}
&
{\tt pll28@cam.ac.uk}
\end{tabular}\\
\\
}

\begin{document}
\maketitle
\begin{abstract}
We study simple binary hypothesis testing under both local differential privacy
(LDP) and communication constraints. We qualify our results as either minimax optimal or instance optimal: the former hold for the set of distribution pairs with prescribed Hellinger divergence and total variation distance, whereas the latter hold for specific distribution pairs. For the sample complexity of simple hypothesis testing under pure LDP constraints, we establish instance-optimal bounds for distributions with binary support; minimax-optimal bounds for general distributions; and (approximately) instance-optimal, computationally efficient algorithms for general distributions. When both privacy and communication constraints are present, we develop instance-optimal, computationally efficient algorithms that achieve the minimum possible sample complexity (up to universal constants). Our results on instance-optimal algorithms hinge on identifying the extreme points of the joint range set $\cA$ of two distributions $p$ and $q$, defined as $\cA := \{(\bT p, \bT q) | \bT \in \cC\}$, where $\cC$ is the set of channels characterizing the constraints.
\end{abstract}

\section{Introduction}
Statistical inference on distributed data is becoming increasingly common, due to the proliferation of massive datasets which cannot be stored on a single server, and greater awareness of the security and privacy risks of centralized data. An institution (or statistician) that wishes to infer an aggregate statistic of such distributed data needs to solicit information, such as the raw data or some relevant statistic, from data owners (individuals). Individuals may be wary of sharing their data due to its sensitive nature or their lack of trust in the institution. The local differential privacy (LDP) paradigm suggests a solution by requiring that individuals' responses divulge only a limited amount of information about their data to the institution. Privacy is typically ensured by deliberately randomizing individuals' responses, e.g., by adding noise. See \cref{def:ldp} below for a formal definition; we refer the reader to Dwork and Roth~\cite{DworkRoth13} for more details on differential privacy.

In this paper, we study distributed estimation under LDP constraints, focusing on simple binary hypothesis testing, a fundamental problem in statistical estimation.
We will also consider LDP constraints in tandem with communication constraints. This is a more realistic setting, since bandwidth considerations often impose constraints on the size of individuals' communications. The case when only communication constraints are present was addressed previously by Pensia, Jog, and Loh~\cite{PenJL22}.

Recall that simple binary hypothesis testing is defined as follows: Let $p$ and
$q$ be two distributions over a finite domain $\cX$, and let $X_1, \dots, X_n \in \cX^n$ be $n$ i.i.d.\ samples drawn from either $p$ or $q$. The goal of the statistician is to identify (with high probability) whether the samples were drawn from $p$ or $q$. This problem has been extensively studied in both asymptotic and nonasymptotic settings~\cite{NeyPea33, Wald45, Cam86}. For example, it is known that the optimal test for this problem is the likelihood ratio test, and its performance can be characterized in terms of divergences between $p$ and $q$, such as the total variation distance, Hellinger divergence, or Kullback--Leibler divergence.
In particular, the sample complexity of hypothesis testing, defined as the
smallest sample size needed to achieve an error probability smaller than a small constant, say, $0.01$, is $\Theta\left(\frac{1}{\hel^2(p,q)}\right)$, where $\hel^2(p,q)$ is the Hellinger divergence between $p$ and $q$.

In the context of local differential privacy, the statistician no longer has access to the original samples $X_1, \dots, X_n$, but only their privatized counterparts: $Y_1, \dots, Y_n \in \cY^n$, for some set $\cY$.\footnote{As shown in Kairouz, Oh, and Viswanath~\cite{KaiOV16}, for simple binary hypothesis testing, we can take $\cY$ to be $\cX$, with the same sample complexity (up to constant factors); see \Cref{fact:id-channel}.} Each $X_i$ is transformed to $Y_i$ via a private channel $\bT_i$, which is simply a probability kernel specifying $\bT_i(y,x) = \mathbb P(Y_i = y|X_i = x)$. With a slight abuse of notation, we shall use $\bT_i$ to denote the transition kernel in $\real^{|{\cY}| \times |{\cX}|}$, as well as the stochastic map $Y_i = \bT_i(X_i)$.
A formal definition of privacy is given below:

\begin{definition}[$\epsilon$-LDP]
\label{def:ldp}
Let $\epsilon \in \R_+$, and let $\cX$ and $\cY$ be two domains. A channel $\bT: \cX \to \cY$	\emph{satisfies $\epsilon$-LDP} if
\begin{equation*}
\sup_{x,x' \in \cX} \sup_{A \subseteq \cY} \P[\bT(x) \in A] - e^{\epsilon} \cdot \P[\bT(x') \in A] \leq 0,
\end{equation*}
where we interpret $\bT$ as a stochastic map on $\cX$.
Equivalently, if $\cX$ and $\cY$ are countable domains (as will be the case for us), a channel $\bT$ is $\epsilon$-LDP if $\sup_{x, x' \in \cX} \sup_{y \in \cY} \frac{\bT(y,x)}{\bT(y,x')} \leq e^\epsilon$, where we interpret $\bT$ as the transition kernel.
\end{definition}
When $\epsilon = \infty$, we may set $Y_i$ equal to $X_i$ with probability 1, and we recover the vanilla version of the problem with no privacy constraints.

Existing results on simple binary hypothesis testing under LDP constraints
have focused on the high-privacy regime of $\epsilon \in (0,c)$, for a constant
$c > 0$, and have shown that the sample complexity is
$\Theta\left(\frac{1}{\epsilon^2 \dtv^2(p,q)}\right)$,
where $\dtv(p,q)$ is the total variation distance between $p$ and
$q$~(cf.~\Cref{fact:sample-complexity-eps-small}).
Thus, when $\epsilon$ is a constant, the sample complexity is $\Theta\left(\frac{1}{\dtv^2(p,q)}\right)$, and when $\epsilon = \infty$ (no privacy), the sample complexity is $\Theta\left(\frac{1}{\hel^2(p,q)}\right)$. Although these two divergences satisfy $\dtv^2(p,q) \lesssim \hel^2(p,q) \lesssim \dtv(p,q)$, the bounds are tight; i.e., the two sample complexities can be quadratically far apart. Existing results therefore do not inform sample complexity when $1 \ll \epsilon < \infty$. This is not an artifact of analysis: the optimal tests in the low and high privacy regimes are fundamentally different.

The large-$\epsilon$ regime has been increasingly used in practice, due to privacy amplification provided by shuffling~\cite{CheSUZZ19, BitEMMRLRKTS17, FelMT21}. Our paper makes progress on the computational and statistical fronts in the large-$\epsilon$ regime, as will be highlighted in \Cref{sec:our-results} below.

\subsection{Problem Setup}

For a natural number $k$, we use $[k]$ to denote the set $\{1,2,\dots,k\}$.
In this paper, we focus on \emph{non-interactive private-coin} protocols. It has been shown in \cite[Appendix A]{PenJL22} that the sample complexity of simple binary hypothesis testing under information constraints is the same (up to multiplicative constants) for  \emph{non-interactive public-coin} and \emph{non-interactive private-coin} protocols.\footnote{We refer the reader to Acharya, Canonne, Liu, Sun, and Tyagi~\cite{AchCLST22-interactive} for differences between various protocols.} As we will be working with both privacy and communication constraints in this
paper, we first define the generic protocol for distributed inference under an
arbitrary set of channels $\cC$ below:

\begin{definition}[Simple binary hypothesis testing under channel constraints]
\label{def:shtgeneric}
Let $\cX$ and $\cY$ be two countable sets. Let $\cC$ be a set of channels from $\cX$ to $\cY$, and let $p$ and $q$ be two distributions on $\cX$. Let $\{U_i\}_{i=1}^n$ denote a set of $n$ users who choose channels $\{\bT_i\}_{i=1}^n \in \cC^n$ according to a deterministic rule\footnote{When $\cC$ is a convex set of channels, as will be the case in this paper, the deterministic rules are equivalent to randomized rules (with independent randomness).} $\cR: [n] \to \cC$.
Each user $U_i$ then observes $X_i$ and generates $Y_i = \bT_i(X_i)$ independently, where $X_1,\dots,X_n$ is a sequence of i.i.d.\ random variables drawn from an (unknown) $r \in \{p,q\}$. The central server $U_0$ observes $(Y_1,\dots,Y_n)$ and constructs an estimate $\widehat{r} = \phi(Y_1,\dots,Y_n)$, for some test $\phi: \cup_{i=1}^\infty \cY^i \to \{p,q\}$. We refer to this problem as \emph{simple binary hypothesis testing under channel constraints}.
\end{definition}
In the non-interactive setup, we can assume that all $\bT_i$'s are identical
equal to some $\bT$, as it will increase the sample complexity by at most a
constant factor~\cite{PenJL22} (cf.\ \Cref{fact:id-channel}).
We now specialize the setting of \Cref{def:shtgeneric} to the case of LDP
constraints:
\begin{definition}[Simple binary hypothesis testing under LDP constraints]
\label{def:sht-ldp}
Consider the problem in \cref{def:shtgeneric} with $\cY = \N$, where $\cC$ is the set of all $\epsilon$-LDP channels from $\cX$ to $\cY$. We denote this problem by $\cB(p,q,\epsilon)$. For a given test-rule pair $(\phi,\cR)$ with $\phi:\cup_{j=1}^ \infty \cY^j \to \{p,q\}$, we say that $(\phi,\cR)$ \emph{solves} $\cB(p,q,\epsilon)$ \emph{with sample complexity $n$} if
\begin{align}
\label{eq:DefErrorProbCommCons}
\P_{(X_1,\dots,X_n) \sim p^{\otimes n}}( \phi(Y_1,\dots,Y_n) \ne p) +   \P_{(X_1,\dots,X_n) \sim q^{\otimes n}}( \phi(Y_1,\dots,Y_n) \ne q) \leq 0.1.
\end{align}
We use $\nstar(p,q,\epsilon)$ to denote the sample complexity of this task, i.e., the smallest $n$ so that there exists a $(\phi,\cR)$-pair that solves $\cB(p,q,\epsilon)$. We use $\cB(p,q)$ and $\nstar(p,q)$ to refer to the setting of non-private testing, i.e., when $\epsilon = \infty$, which corresponds to the case when $\cC$ is the set of all possible channels from $\cX$ to $\cY$.
\end{definition}

For any fixed rule $\cR$, the optimal choice of $\phi$ corresponds to the likelihood ratio test on $\{Y_i\}_{i=1}^n$. Thus, in the rest of this paper, our focus will be optimizing the rule $\cR$, with the choice of $\phi$ made implicitly. In fact, we can take $\cY$ to be $\cX$, at the cost of a constant-factor increase in the sample complexity~\cite{KaiOV16} (cf.~\Cref{fact:output-sze-same}).

We now define the threshold for free privacy, in terms of a large enough universal constant $C_{\text{thresh}} > 0$ which can be explicitly deduced from our proofs:
\begin{definition}[Threshold for free privacy]
\label{def:free-privacy}
We define $\epsilon^*(p,q)$ (also denoted by $\epsilon^*$ when the context is clear) to be the smallest $\epsilon$ such that $\nstar(p,q,\epsilon) \leq C_{\text{thresh}} \cdot \nstar(p,q)$; i.e., for all $\epsilon \ge \epsilon^*(p,q)$, we can obtain 	$\epsilon$-LDP without any substantial increase in sample complexity compared to the non-private setting.
\end{definition}

Next, we study the problem of simple hypothesis testing under both privacy and
communication constraints.
By communication constraints,
we mean that the channel $\bT$ maps from $\cX$ to $[\ell]$
for some $\ell \in \N$, which is potentially much smaller than $|\cX|$.

\begin{definition}[Simple binary hypothesis testing under LDP and communication constraints]
\label{def:sht-ldp-comm}
Consider the problem in \cref{def:shtgeneric} and \cref{def:sht-ldp}, with $\cC$ equal to the set of all channels that satisfy $\epsilon$-LDP and $\cY = [\ell]$. We denote this problem by $\cB(p,q,\epsilon,\ell)$, and use $\nstar(p,q,\epsilon, \ell)$ to denote its sample complexity.
\end{definition}

Communication constraints are worth studying not only for their practical relevance in distributed inference, but also for their potential to simplify algorithms without significantly impacting performance. Indeed, the sample complexities of simple hypothesis testing with and without communication constraints are almost identical \cite{BhaNOP21,PenJL22} (cf.~\cref{fact:comm-constraints}), even for a single-bit ($\ell = 2$) communication constraint. As we explain later, a similar statement can be made for privacy constraints, as well. %

\subsection{Existing Results}
\label{sec:existing results}

As noted earlier, the problem of simple hypothesis testing with just communication constraints was addressed in Pensia, Jog, and Loh~\cite{PenJL22}. Since communication and privacy constraints are the most popular information constraints studied in the literature, the LDP-only and LDP-with-communication-constraints settings considered in this paper are natural next steps. Many of our results, particularly those on minimax-optimal sample complexity bounds, are in a similar vein as those in Pensia, Jog, and Loh~\cite{PenJL22}. Before describing our results, let us briefly mention the most relevant prior work. We discuss further related work in \Cref{sec:related_work}. 

\paragraph{Existing results on sample complexity.}
Existing results (cf.\ Duchi, Jordan, and Wainwright~\cite[Theorem 1]{DucJW18} and Asoodeh and Zhang~\cite[Theorem 2]{AsoZha22}) imply that
\begin{align}
    \label{eq:existing-samp-compl}
    \nstar(p,q,\epsilon) \gtrsim  \begin{cases}
        \frac{1}{\epsilon^2 \,\cdot\, \dtv^2(p,q)}, & \text{if } \epsilon \in (0,1], \\
        \frac{1}{e^{\epsilon} \,\cdot\,  \dtv^2(p,q) },  & \text{if } e^\epsilon \in \left(e, \frac{\hel^2(p,q)}{\dtv^2(p,q)}\right], \\
        \frac{1}{\hel^2(p,q) },  & \text{if } e^\epsilon > \frac{\hel^2(p,q)}{\dtv^2(p,q)}.
    \end{cases}
\end{align}
An upper bound on the sample complexity can be obtained by choosing a specific
private channel $\bT$ and analyzing the resulting test.
A folklore result~(see, for example, Joseph, Mao, Neel, and Roth~\cite[Theorem 5.1]{JosMNR19}) shows that
setting $\bT = \bT_{\text{RR}} \times \bT_{\text{Scheffe}}$, where
$\bT_{\text{Scheffe}}$ maps $\cX$ to $\{0,1\}$ using a threshold rule based on
$\frac{p(x)}{q(x)}$, and $\bT_{\text{RR}}$ is the binary-input
binary-output randomized response channel, gives $\nstar(p,q,\epsilon) \lesssim
	\frac{1}{\min(1, \epsilon^2)\, \cdot\, \dtv^2(p,q)}$.
This shows that when $\epsilon \in (0,1]$ (or $(0, c]$, for some constant $c$),
the lower bound is tight up to constants. Observe that for any $\ell \ge 2$, the sample complexity with privacy and communication constraints $\nstar(p, q, \epsilon, \ell)$ also satisfies the same lower and upper bounds, since the channel $\bT$ has only two outputs.%
 
However, the following questions remain unanswered:
\begin{quotation}
	\begin{center}
		\emph{What is the optimal sample complexity for $\epsilon \gg 1$?
			In particular, are the existing lower bounds~\Cref{eq:existing-samp-compl}
			tight?
			What is the threshold for free privacy?
		}
	\end{center}
\end{quotation}
In \Cref{ssub:our_results_statistical_rates}, we establish minimax-optimal bounds on the sample complexity for all values of $\epsilon$, over sets of distribution pairs with fixed total variation distance and Hellinger divergence. In particular, we show that the lower bounds~\Cref{eq:existing-samp-compl} are tight for binary distributions, but may be arbitrarily loose for general distributions.

\paragraph{Existing results on computationally efficient algorithms.}
Recall that each user needs to select a channel $\bT$ to optimize the sample complexity. Once $\bT$ is chosen, the optimal test is simply a likelihood ratio test between $\bT p$ and $\bT q$. Thus, the computational complexity lies in determining $\bT$. As noted earlier, for $\epsilon\leq 1$, the optimal channel is $\bT = \bT_{\text{RR}} \times \bT_{\text{Scheffe}}$, and this can be computed efficiently. However, this channel $\bT$ may no longer be optimal in the regime of $\epsilon \gg 1$.

As with statistical rates, prior literature on finding optimal channels for $\epsilon\gg1$ is scarce. Existing algorithms either take time exponential in the domain size~\cite{KaiOV16}, or their sample complexity is suboptimal by polynomial factors (depending on $\frac{1}{\dtv^2(p,q)}$, as opposed to $\frac{1}{\hel^2(p,q)}$).
This raises the following natural question:
\begin{quotation}
\begin{center}
\emph{Is there a polynomial-time algorithm that finds a channel $\bT$ whose sample complexity is (nearly) optimal?}
\end{center}
\end{quotation}
We answer this question in the affirmative in \Cref{ssub:results_comp}.

\subsection{Our Results}
\label{sec:our-results}

We are now ready to describe our results in this paper, which we outline in the next three subsections. In particular, \Cref{ssub:our_results_statistical_rates} focuses on the sample complexity of simple hypothesis testing under local privacy, \Cref{ssub:our_results_extreme} focuses on structural properties of the extreme points of the joint range under channel constraints, and \Cref{ssub:results_comp} states our algorithmic guarantees.

\subsubsection{Statistical Rates}
\label{ssub:our_results_statistical_rates}

We begin by analyzing the sample complexity when both $p$ and $q$ are binary distributions. We prove the following result in \Cref{sec:binary-dist-samp-comp}, showing that the existing lower bounds~\Cref{eq:existing-samp-compl} are tight for binary distributions:

\begin{restatable}[Sample complexity of binary distributions]{theorem}
{LemBinaryRRSampleComplexity}
\label{lem:binary-samp-comp}
Let $p$ and $q$ be two binary distributions. Then
\begin{align}
\label{eq:sample-comp-binary}
\nstar(p,q,\epsilon ) & \asymp
\begin{cases}
\frac{1}{\epsilon^2 \,\cdot\, \dtv^2(p,q)}, & \text{if } \epsilon \leq 1, \\
\frac{1}{e^\epsilon \,\cdot\,\dtv^2(p,q)}, & \text{if } e^\epsilon \in \left[e, \frac{\hel^2(p,q)}{\dtv^2(p,q)} \right], \\
\frac{1}{\hel^2(p,q)}, & \text{if } e^\epsilon >
\frac{\hel^2(p,q)}{\dtv^2(p,q)}.
\end{cases}
\end{align}
\end{restatable}
In particular, the threshold $\epsilon^*$ for free privacy (\Cref{def:free-privacy}) satisfies $e^{\epsilon^*}\asymp \frac{\hel^2(p,q)}{\dtv^2(p,q)}$.
Note that the sample complexity $\nstar(p, q, \epsilon)$ for all ranges of $\epsilon$ is completely characterized by the total variation distance and Hellinger divergence between $p$ and $q$. A natural set to consider is \emph{all distribution pairs} (not just those with binary support) with a prescribed total variation distance and Hellinger divergence; we investigate minimax-optimal sample complexity over this set. Our next result shows that removing the binary support condition radically changes the sample complexity, even if the total variation distance and Hellinger divergence are the same. Specifically, we show that there are ternary distribution pairs whose sample complexity (as a function of the total variation distance and Hellinger divergence) is significantly larger than the corresponding sample complexity for binary distributions. %
\begin{restatable}[Sample complexity lower bound for general distributions]{theorem}{LemSampCompWorstCase}
\label{lem:worst-case-samp-comp} For any $\rho \in (0,0.5)$ and $\nu \in (0,0.5)$ such that $2\nu^2 \leq \rho \leq \nu$%
, there exist ternary distributions $p$ and $q$ such that $\hel^2(p,q) = \rho$, $\dtv(p,q) = \nu$,
and the sample complexity behaves as
\begin{align}
\label{eq:worst-case-samp-comp}
\nstar(p,q,\epsilon) \asymp
\begin{cases}
\frac{1}{\epsilon^2\,\cdot\, \dtv^2(p,q)}, & \text{ if } \epsilon \leq 1, \\ \min\left( \frac{1}{\dtv^2(p,q)}, 
\frac{1}{e^\epsilon \,\cdot\, \hel^4(p,q)} \right), & \text{ if } e^\epsilon \in \left[e,
\frac{1}{\hel^2(p,q)}\right], \\ \frac{1}{\hel^2(p,q)}, & \text{ if } e^\epsilon >
\frac{1}{\hel^2(p,q)}.
\end{cases}
\end{align}
\end{restatable}
We prove this result in \Cref{sec:worst-case-samp-comp-lwr-bds}.
\begin{remark}
\label{rem:binary-vs-worst}
We highlight the differences between the sample complexity in the binary setting (cf.\ equation~\Cref{eq:sample-comp-binary}) and the worst-case general distributions (cf.\ equation~\Cref{eq:worst-case-samp-comp}) below (also see \Cref{fig:sample-complexity}):

\begin{enumerate}
\item (Relaxing privacy may not lead to significant improvements in accuracy.) In equation~\Cref{eq:worst-case-samp-comp}, there is an arbitrarily large range of $\epsilon$ where the sample complexity remains roughly constant.
In particular, when $e \leq e^\epsilon \lesssim \frac{\dtv^2(p,q)}{\hel^4(p,q)}$, the sample complexity of hypothesis testing remains roughly the same (up to constants). That is, we are sacrificing privacy without any significant gains in statistical efficiency. This is in stark contrast to the binary setting, where increasing $e^\epsilon$ by a large constant factor leads to a constant-factor improvement in sample complexity.
\item (The threshold for free privacy is larger.) Let $\epsilon^* := \epsilon(p,q)$ be the threshold for free privacy (cf. \Cref{def:free-privacy}). In the binary setting, one has $e^{\epsilon^*} \asymp \frac{\hel^2(p,q)}{\dtv^2(p,q)}$, whereas for general distributions, one may need $e^{\epsilon^*} \gtrsim \frac{1}{\hel^2(p,q)}$. The former $\epsilon^*$ can be arbitrarily smaller than the latter.
\end{enumerate}
\end{remark}

To complement the result above, which provides a lower bound on the sample complexity for worst-case distributions, our next result provides an upper bound on the sample complexity that nearly
matches the rates (up to logarithmic factors) for arbitrary distributions. Moreover, the proposed algorithm uses an $\epsilon$-LDP channel with \emph{binary} outputs. The following result is proved in \Cref{sec:gen-dist-samp-comp-upr-bds}:

\begin{restatable}[Sample complexity upper bounds and an efficient algorithm for
		hypothesis testing for general distributions]
	{theorem}{LemUppBoundCommEff}
	\label{thm:samp-comp-uprbds}
	Let $p$ and $q$ be two distributions on $[k]$.
	Let $\epsilon > 0$.
	Then the sample complexity behaves as
	\begin{align}
		\nstar(p,q,\epsilon) &\lesssim
		\begin{cases}
			\frac{1}{\epsilon^2\,\cdot\, \dtv^2(p,q)}, & \text{ if } \epsilon \leq 1, \\ \min\left(
			\frac{1}{\dtv^2(p,q)}, \frac{
					\alpha^2}{e^\epsilon \,\cdot\, \hel^4(p,q)} \right), & \text{ if } e^\epsilon \in \left(e,
				\frac{\alpha}{\hel^2(p,q)}\right], \\
				\frac{\alpha}{\hel^2(p,q)}, & \text{ if } e^\epsilon
			> \frac{\alpha}{\hel^2(p,q)},
		\end{cases}
	\end{align}
	where $\alpha \lesssim \log(1/\hel^2(p,q)) \asymp \log\left( \nstar(p,q)
		\right)$.

	Moreover, the rates above are achieved by an $\epsilon$-LDP channel $\bT$ that
	maps $[k]$ to $[2]$ and can be found in time polynomial in $k$, for any
	choice of $p$, $q$, and $\epsilon$.
\end{restatable}

\Cref{lem:worst-case-samp-comp,thm:samp-comp-uprbds} imply that the above sample complexity is minimax optimal (up to logarithmic factors) over the class of distributions with total variation distance $\nu$ and Hellinger divergence $\rho$ satisfying the conditions in \Cref{lem:worst-case-samp-comp}. We summarize this in the following theorem:

\begin{theorem}[Minimax-optimal bounds]
\label{thm:minmax-optimal-bounds}
Let $\rho \in (0,0.5)$ and $\nu \in (0,0.5)$ be such that $2\nu^2 \leq \rho \leq \nu$. Let $S_{\rho, \nu}$ be the set of all distribution pairs with discrete supports, with total variation distance and Hellinger divergence being $\nu$ and $\rho$, respectively:
$$S_{\rho, \nu} := \{(p, q): k \in \mathbb N, p \in \Delta_k, q \in \Delta_k, \dtv(p,q) = \nu, \hel^2(p,q) = \rho\}.$$ 
Let $\nstar(S_{\rho, \nu}, \epsilon)$  be the minimax-optimal sample complexity of hypothesis testing under $\epsilon$-LDP constraints, defined as 
\begin{align*}
\nstar(S_{\rho, \nu}, \epsilon) = \min_{(\phi, \cR)} \max_{(p, q) \in S_{\rho, \nu}} \nstar(p, q, \epsilon),
\end{align*}
for test-rule pairs $(\phi, \cR)$, as defined in \Cref{def:sht-ldp}. Then
\begin{align}
\label{eq:minimax-optimal-samp-comp}
\nstar(S_{\rho, \nu},\epsilon) = 
\begin{cases}
\widetilde \Theta \left(\frac{1}{\epsilon^2\,\cdot\, \nu^2} \right), & \text{ if } \epsilon \leq 1, \\
\widetilde \Theta \left(\min\left( \frac{1}{\nu^2}, 
\frac{1}{e^\epsilon \,\cdot\, \rho^2} \right) \right), & \text{ if } e^\epsilon \in \left[e,
\frac{1}{\rho}\right], \\ 
\widetilde \Theta \left(\frac{1}{\rho} \right), & \text{ if } e^\epsilon >
\frac{1}{\rho}.
\end{cases}
\end{align}
Here, the $\widetilde \Theta$ notation hides poly-logarithmic factors in $1/\nu$ and $1/\rho$. 
\end{theorem}

\begin{remark}
A version of the above theorem may also be stated for privacy and communication constraints, by defining 
\begin{align*}
\nstar(S_{\rho, \nu}, \epsilon, \ell) = \min_{(\phi, \cR)} \max_{(p, q) \in S_{\rho, \nu}} \nstar(p, q, \epsilon, \ell).
\end{align*}
In fact, it may seen that the same sample complexity bounds continue to hold for $\nstar(S_{\rho, \nu}, \epsilon, \ell)$, with $\ell \ge 2$,
since the lower bound in \Cref{lem:worst-case-samp-comp} continues to hold with communication constraints, as does the upper bound in \Cref{thm:samp-comp-uprbds}, which uses a channel with only binary outputs. %
\end{remark}

\begin{remark}
The above theorem mirrors a minimax optimality result for communication-constrained hypothesis testing from Pensia, Jog, and Loh~\cite{PenJL22}. There, the set under consideration was $S_\rho$, where $\rho$ is the Hellinger divergence between the distribution pair, and the minimax-optimal sample complexity was shown to be $ \widetilde \Theta (1/\rho)$ even for a binary communication constraint. 
\end{remark}

Finally, we consider the { threshold for free privacy} $\epsilon^*$ for general
distributions; see \Cref{def:free-privacy}. Observe that \Cref{thm:samp-comp-uprbds} does not provide any upper bounds on $\epsilon^*$, since the sample complexity in \Cref{thm:samp-comp-uprbds} is bounded away from $\nstar(p,q)$, due to the logarithmic multiplier $\alpha$.
Recall that \Cref{lem:worst-case-samp-comp} implies $e^{\epsilon^*} \gtrsim \frac{1}{\hel^2(p,q)}$ in the worst case. Our next result, proved in \Cref{sec:gen-dist-samp-comp-upr-bds}, shows that
this is roughly tight, and $e^{\epsilon^*} \lesssim \frac{1}{\hel^2(p,q)} \cdot
	\log\left(\frac{1}{\hel^2(p,q)}\right)$ for all distributions:
\begin{restatable}[]{theorem}{LemFreePrivacy}
	\label{lem:privacy-free}
	Let $p$
	and $q$ be two distributions on $[k]$,
	and let $e^\epsilon \gtrsim \frac{1}{\hel^2(p,q)}
		\log\left(\frac{1}{\hel^2(p,q)}\right)$.
	Then $\nstar(p,q,\epsilon) \asymp \nstar(p,q)$.
	Moreover, there is a channel $\bT$ achieving this sample complexity that maps
	$[k]$ to a domain of size $\lceil \log(\nstar(p,q)) \rceil$, and which can be
	computed in $\poly(k, \log(\lceil \nstar(p,q) \rceil))$ time.
\end{restatable}

We thereby settle the question of minimax-optimal sample complexity (up to logarithmic factors) for simple binary hypothesis testing under LDP-only and LDP-with-communication constraints (over the class of distributions with a given total variation distance and Hellinger divergence). Moreover, the minimax-optimal upper bounds are achieved by computationally efficient, communication-efficient algorithms. However, there can be a wide gap between instance-optimal and minimax-optimal procedures; in the next two subsections, we present structural and computational results for instance-optimal algorithms.

\begin{figure}[H]
    \includegraphics[width=0.8\textwidth]{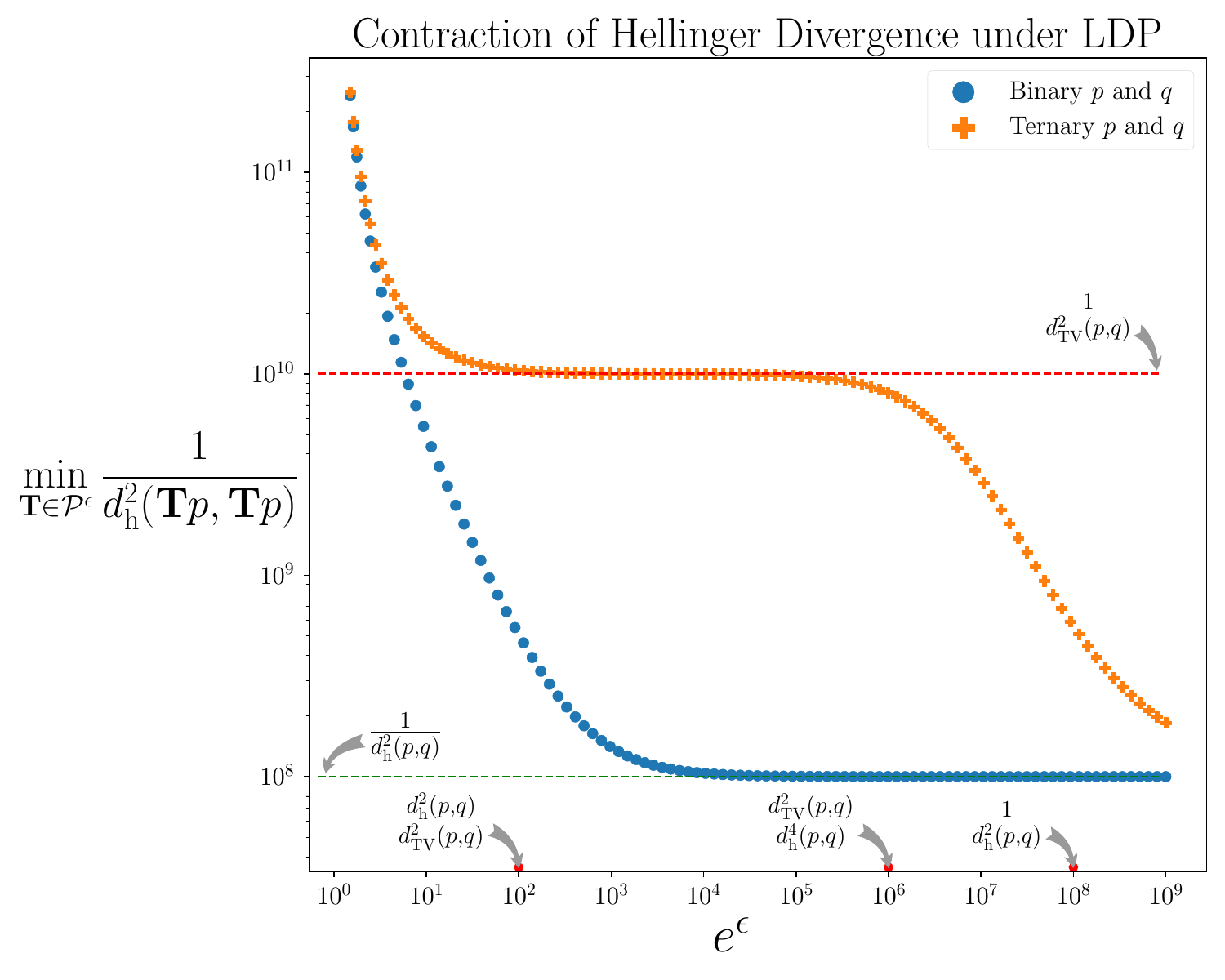}
    \caption{
    In this plot, we show the difference between the behavior of sample complexity under $\epsilon$-LDP constraints for binary distributions and (worst-case) ternary distributions from \Cref{lem:worst-case-samp-comp}.
    We take two pairs of distributions $(p,q)$---one pair of binary distributions (shown in blue, with marker $\circ$) and one pair of ternary distributions (shown in orange, with marker $+$)---such that the two pairs have Hellinger divergence $\hel^2(p,q) = 10^{-8}$ and total variation distance $\dtv(p,q) = 10^{-5}$.
    For each value of $\epsilon$, shown on the horizontal axis after being mapped to $e^\epsilon$, we compute $\min_{\bT \in \cP^\epsilon} 1/\hel^2(\bT p, \bT q)$, where $\cP^\epsilon$ is the set of all $\epsilon$-LDP channels, and plot it on vertical axis.
    Thus, the vertical axis characterizes the sample complexity $\nstar(p,q,\epsilon)$ of simple binary hypothesis testing between $p$ and $q$ with privacy constraints, up to constant factors (cf.~\Cref{fact:id-channel}). 
    Both axes are shown in $\log$-scale here. 
    Since the total variation distance between the two pairs is identical, we see that their curves overlap for small $\epsilon$ ($\epsilon \ll1$, which is consistent with the fact that $\nstar(p,q,\epsilon) \asymp \frac{1}{\epsilon^2 \dtv^2(p,q)}$ for small $\epsilon$).
    As predicted by \Cref{lem:binary-samp-comp}, the curve for binary distributions decreases rapidly for $\epsilon\gg 1$ until it saturates at $1/\hel^2(p,q)$. 
    Moreover, for $e^{\epsilon}\asymp \hel^2(p,q)/\dtv^2(p,q)$, the predicted threshold for free privacy, the vertical axis is within constant factors of its asymptotic value, as predicted.
    On the other hand, the curve for ternary distributions seems to have three different phases, as predicted by \Cref{lem:worst-case-samp-comp}: (i) for small $\epsilon$, it behaves as $1/(\epsilon^2\dtv^2(p,q))$; (ii) for moderate values of $\epsilon$,  such that $e \ll e^{\epsilon} \ll \frac{\dtv^2(p,q)}{\hel^4(p,q)}$, it remains stagnant roughly at $\frac{1}{\dtv^2(p,q)}$; and (iii) for $e^\epsilon \gg \frac{\dtv^2(p,q)}{\hel^4(p,q)}$, the curve decreases rapidly until it approaches $1/\hel^2(p,q)$.
    The phase (ii) corresponds to the phenomenon that we are leaking privacy without any gains in statistical efficiency.
    Finally, $e^\epsilon$ needs to be as large as $1/\hel^2(p,q)$ for the vertical axis to be within a factor of $10$ of its asymptotic value.
    We refer the reader to \Cref{rem:binary-vs-worst} for more details.} 
	\label{fig:sample-complexity}
\end{figure}
\subsubsection{Structure of Extreme Points under the Joint Range}
\label{ssub:our_results_extreme}

In this section, we present results for the extreme points of the joint range
of an arbitrary pair of distributions when transformed by a set of channels.
Formally, if $\cC$ is a convex set of channels from $\cX$ to $\cY$, and $p$ and
$q$ are two distributions on $\cX$, we are interested in the extreme points of
the set $\cA := \{(\bT p, \bT q): \bT \in \cC\}$, which is a convex subset of
$\Delta_{|\cY|} \times \Delta_{|\cY|}$.\new{\footnote{For $k \in \N$, we use $\Delta_{k}$ to denote the probability simplex on a domain of alphabet size $k$.}} Recall that a point $a \in \cA$ is said to be an extreme point if $a$ cannot be expressed as a convex combination of two distinct points in $\cA$; i.e., if $a = \lambda b + (1-\lambda) c$ for $\lambda \in (0,1)$, then $b=c~(=a)$. The extreme points of a convex set are naturally insightful for maximizing
quasi-convex functions, and we will present the consequences of the results in
this section in \Cref{ssub:results_comp}.

We consider two choices of $\cC$: first, when $\cC$ is the set of all channels from $\cX$ to $\cY = [\ell]$, and second, when $\cC$ is the set of all $\epsilon$-LDP channels from $\cX$ to $\cY = [\ell]$. We use $\cT_{\ell,k}$ to denote the set of all channels that map from $[k]$ to $[\ell]$.

The following class of deterministic channels plays a critical role in our theory:

\begin{definition}[Threshold channels]
	\label{def:thresh-channel}
	For some $k \in \N$, let $p$ and $q$ be
	two distributions on $[k]$.
	For any $\ell \in \N$, a deterministic channel $\bT \in \cT_{\ell,k}$ is
	a \emph{threshold channel} if the following property holds for every $u, v \in [k]$: If
	$\frac{p(u)}{q(u)} < \frac{p(v)}{q(v)}$ and $\bT(u) = \bT(v)$, then any $w \in [k]$
	such that $\frac{p(w)}{q(w)} \in \left(\frac{p(u)}{q(u)}, \frac{p(v)}{q(v)}
		\right)$ satisfies $\bT(w) = \bT(u) (= \bT(v))$.
	(The likelihood ratios are assumed to take values on the extended real line; i.e., on $\real \cup \{-\infty, +\infty\}$.)
\end{definition}

\begin{remark}
	Threshold channels are intuitively easy to understand when all the likelihood
	ratios are distinct (this may be assumed without loss of generality in our
	paper, as explained later): Arrange the inputs in increasing order of their
	likelihood ratios and partition them into $\ell$ contiguous blocks.
   Thus, there are at most $k^\ell$ such threshold channels (up to reordering of output labels).
\end{remark}

Our first result proved in \Cref{sec:ext-points-comm} is for the class of
communication-constrained channels, and shows that all extreme points of the
joint range are obtained using deterministic threshold channels:
\begin{restatable}[Extreme points of the joint range under
		communication constraints] {theorem}{ThmJointRangeComm}
	\label{thm:ext-point-comm}
	Let $p$ and $q$ be two distributions on $[k]$.
	Let $\cA$ be the set of all pairs of distributions that are obtained by passing
	$p$ and $q$ through a channel of output size $\ell$, i.e.,
	\begin{align*}
		\cA =
		\{(\bT p, \bT q): \bT \in \cT_{\ell,k}\}.
	\end{align*}
	If $(\bT p, \bT q)$ is an extreme point of $\cA$, then $\bT$ is a threshold
	channel.
\end{restatable}
We note that the above result is quite surprising: $(\bT p, \bT q)$ is extreme point of $\cA$ only if $\bT$ is an extreme point of $\cT_{\ell,k}$ (i.e., a deterministic channel), but \Cref{thm:ext-point-comm} demands that $\bT$ be a deterministic \emph{threshold} channel, meaning it lies in a very small subset of deterministic channels. Indeed, even for $\ell = 2$, the number of deterministic channels from $[k]$ to $[2]$ is $2^k$, whereas the number of threshold channels is just $2k$.
We note that the result above is similar in spirit to Tsitsiklis~\cite[Proposition 2.4]{Tsitsiklis93}.
However, the focus there was on a particular objective, the probability of
error in simple hypothesis testing, with non-identical channels for users. Our result is for identical channels and is generally applicable to
quasi-convex objectives, as mentioned later.

We now consider the case where $\cC$ is the set of $\epsilon$-LDP channels from $[k]$ to $[\ell]$. Since $\cC$ is a set of private channels, it does not contain any deterministic channels (thus, does not contain threshold channels). Somewhat surprisingly, we still show that the threshold channels play a fundamental role in the extreme points of the joint range under $\cC$. The following result shows that any extreme point of the joint range $\cA$ can be obtained by a threshold channel mapping into $[2\ell^2]$, followed by an $\epsilon$-LDP channel from $[2\ell^2]$ to $[\ell]$:

\begin{restatable}[Extreme points of the joint range under privacy and
		communication constraints]{theorem}{ThmOptStructurePriv}
	\label{thm:ext-point-priv-comm}
	Let $p$ and $q$ be distributions on $[k]$.
	Let $\cC$ be the set of $\epsilon$-LDP channels from $[k]$ to $[\ell]$.
	Let $\cA$ be the set of all pairs of distributions that are obtained by
	applying a channel from $\cC$ to $p$ and $q$, i.e.,
	\begin{align}
		\cA = \{(\bT
		p, \bT q) \mid \bT \in \cC\}.
	\end{align}
	If $(\bT p, \bT q)$ is an extreme point of $\cA$ for $\bT \in \cC$, then $\bT$
	can be written as
	$\bT = \bT_2 \times \bT_1$ for some threshold channel $\bT_1 \in
		\cT_{2\ell^2,k}$ and some $\bT_2$ an extreme point of the set of $\epsilon$-LDP
	channels from $[2\ell^2]$ to $[\ell]$.
\end{restatable}
We prove this structural result in \Cref{sec:ext-points-priv-comm}, which leads to polynomial-time algorithms for constant $\ell$ for maximizing quasi-convex functions, as mentioned in \Cref{ssub:results_comp}.

\subsubsection{Computationally Efficient Algorithms for Instance Optimality}
\label{ssub:results_comp}

The results from the previous sections characterized the minimax-optimal sample
complexity, but did not address instance optimality.
Instance-optimal performance may be substantially better than minimax-optimal performance, as seen by comparing the instance-optimal bounds for binary distributions to the minimax-optimal bounds for general distributions. In this section, we focus on identifying an instance-optimal channel $\bT$ (satisfying the necessary constraints) for a given pair $(p,q)$ of distributions.

Let $p$ and $q$ be fixed distributions over $[k]$. Let $\cP_{\ell, k}^\epsilon$ be the set of all $\epsilon$-LDP channels from $[k]$ to $[\ell]$, and let $\cT_{\ell, k}$ be the set of all channels from $[k]$ to $[\ell]$. Let $\cC \in \{\cP_{\ell, k}^\epsilon, \cT_{\ell, k} \}$. As before, define $\cA = \{(\bT p, \bT q): \bT \in \cC\}$. Let $g: \cA \to \real$ be a (jointly) quasi-convex function; i.e., for all $t \in \real$, the sublevel sets $\{(p', q'): g(p', q') \le t\}$ are convex. In this paper, we are primarily interested in functions corresponding to divergences between the distribution pair. So, unless otherwise mentioned, we shall assume the quasi-convex functions $g$ in this paper are permutation-invariant; i.e., $g(p', q') = g(\Pi p, \Pi q)$ for all permutation matrices $\Pi$. However, our algorithmic results will continue to hold even without this assumption, with an additional factor of $\ell!$ in the time complexity.
We will consider the problem of identifying $\bT$ that solves
\begin{align*}
\max_{\bT \in \cC} g(\bT p, \bT q).
\end{align*}
The quasi-convexity of $g$ implies that the maximum is attained at some $\bT$ such that $(\bT p,\bT q)$ is an extreme point of $\cA$. We can thus leverage the results from \Cref{ssub:our_results_extreme} to search over the subset of channels satisfying certain structural properties.

Identifying $\bT$ that maximizes the Hellinger divergence leads to an \emph{instance-optimal} test for minimizing sample complexity for testing between $p$ and $q$ with channel constraints $\cC$: This is because if each user chooses the channel $\bT$, the resulting sample complexity will be $\Theta\left(\frac{1}{\hel^2(\bT p, \bT q)}\right)$. Thus, the instance-optimal sample complexity will be obtained by a channel $\bT$ that attains $\max_{\bT \in \cC} \hel^2(\bT p, \bT q)$.
Note that the Hellinger divergence is convex (and thus quasi-convex) in its arguments. Apart from the Hellinger divergence, other functions of interest such as the Kullback--Leibler divergence or Chernoff information (which are also convex) characterize the asymptotic error rates in hypothesis testing, so finding $\bT$ for these functions identifies instance-optimal channels in the asymptotic (large-sample) regime. Other potential functions of interest include R\'{e}nyi divergences of all orders, which are quasi-convex, but not necessarily convex~\cite{ErvHar14}.

As mentioned earlier, the results of Kairouz, Oh, and Viswanath~\cite{KaiOV16} give a linear program with $2^k$ variables to find an instance-optimal channel under privacy constraints, which is computationally prohibitive. It is also unclear if their result extends when the channels are further restricted to have communication constraints in addition to privacy constraints.
We now show how to improve on the guarantees of Kairouz, Oh, and Viswanath~\cite{KaiOV16} in the presence
of communication constraints, using the structural results from the previous subsection.

\begin{restatable}[Computationally efficient algorithms for maximizing quasi-convex functions]{corollary}{CorOptimizationPoly}
	\label{cor:alg-priv-quasi-cvx}
	Let $p$ and $q$ be fixed distributions over $[k]$, let
	$\cC \in \{\cT_{\ell,k}, \cP_{ \ell, k}^\epsilon\}$, and let $\cA = \{(\bT p, \bT
		q): \bT \in \cC\}$.
	Let $g: \cA \to \R$ be a jointly quasi-convex function.
	When $\cC = \cT_{\ell,k}$, there is an algorithm that solves
	$\max_{\bT \in \cC} g(\bT p ,\bT q)$
	in
	time polynomial in $k^{\ell}$.
	When $\cC = \cP_{ \ell, k}^\epsilon$, there is an algorithm that solves
	$\max_{\bT \in \cC} g(\bT p ,\bT q)$
	in
	time polynomial in $k^{\ell^2}$ and $2^{\ell^3 \log \ell}$.
\end{restatable}
We prove \Cref{cor:alg-priv-quasi-cvx}  in \Cref{sec:ext-points-comm} and \Cref{SecLDPCommTesting} for $\cC = \cT_{\ell,k}$ and $\cC = \cP_{\ell,k}^\epsilon$, respectively.

\begin{remark}
	When $\ell$ is constant, we obtain a polynomial-time algorithm for maximizing
	any quasi-convex function under $\cT_{\ell,k}$ or $\cP_{ \ell, k}^\epsilon$
	channel constraints.
	When $\cC = \cT_{\ell,k}$ and $g$ is the Kullback--Leibler divergence, this exactly solves (for small $\ell$) %
	a problem introduced in Carpi, Garg, and Erkip~\cite{CarGE21}, which proposed a polynomial-time heuristic.
\end{remark}

Applying the above result to the Hellinger divergence $\hel^2$, we obtain the following
result for simple binary hypothesis testing, proved in \Cref{SecLDPCommTesting}:
\begin{restatable}[Computationally efficient algorithms for instance-optimal
		results under communication constraints]{corollary}{ThmPolyTimeSBHT}
	\label{cor:alg-priv-sbht}
	Let $p$ and $q$ be two distributions on $[k]$.
	For any $\epsilon$ and any integer $\ell > 1$,
	there is an algorithm that runs in time polynomial in $k^{\ell^2}$ and
	$2^{\ell^3 \log \ell}$ and outputs an $\epsilon$-LDP channel $\bT$ mapping from $[k]$ to
	$[\ell]$, such that if $N$ denotes the sample
	complexity of hypothesis testing between $p$ and $q$ when each individual uses
	the channel $\bT$, then $ N \asymp \nstar(p,q,\epsilon, \ell)$.

	In particular, the sample complexity with $\bT$ satisfies
	\begin{align}
     N \lesssim \nstar(p,q,\epsilon)  \cdot \left(1 + \frac{\log\left(\nstar\left(p,q,\epsilon\right) \right)}{\ell} \right).
    \end{align}
The channel $\bT$ may be decomposed as a deterministic threshold channel to a domain of size
	$[2\ell^2]$, followed by an $\epsilon$-LDP channel from $[2\ell^2]$ to $[\ell]$.
\end{restatable}
Thus, by choosing $\ell =2$, we obtain a polynomial-time algorithm with nearly instance-optimal sample complexity (up to logarithmic factors) under just $\epsilon$-LDP constraints.

\subsection{Related Work} %
\label{sec:related_work}

Distributed estimation has been studied extensively under resource constraints
such as memory, privacy, and communication.
Typically, this line of research considers problems of interest such as
distribution estimation~\cite{RobToo70, LeiRiv86, CheKO21, BarHO20}, identity
or independence testing~\cite{AchCT20-I, AchCT20-II, AchCT20-III}, and
parameter estimation~\cite{Hel74,DucJWZ14,  DucJW18, BravGMNW16,
	DucRog19, BarCO20, DiaKPP22-streaming}, and identifies minimax-optimal bounds on the error or
sample complexity. In what follows, we limit our discussion to related work on hypothesis testing
under resource constraints.

For \emph{memory-constrained hypothesis testing}, the earliest works in Cover~\cite{Cover69} and Hellman and Cover~\cite{HelCov73} derived tight bounds on the memory size needed to perform asymptotically error-free testing. Hellman and Cover~\cite{HelCov71} also highlighted the benefits of randomized algorithms. These benefits were also noted in recent work by Berg, Ordentlich, and Shayevitz~\cite{BerOS20}, which considered the error exponent in terms of the memory size. Recently, Braverman, Garg, and Zamir~\cite{BraGZ22} showed tight bounds on the
memory size needed to test between two Bernoulli distributions.

\emph{Communication-constrained hypothesis testing} has two different interpretations.
In the information theory literature, Berger~\cite{Ber79}, Ahlswede and
Csisz{\'a}r~\cite{AhlCsi86}, and Amari and Han~\cite{AmaHan98} considered a
family of problems where two nodes, one which only observes $X_i$'s and the
other which only observes $Y_i$'s, try to distinguish between $P_{XY}$ and $Q_{XY}$. Communication between the nodes occurs over rate-limited channels.
The second interpretation, also called ``decentralized detection'' in
Tsitsiklis~\cite{Tsitsiklis88}, is more relevant to this work.
Here, the observed $X_i$'s are distributed amongst different nodes (one
observation per node) that communicate a finite number of messages (bits) to a
central node, which needs to determine the hypothesis.
Tsitsiklis~\cite{Tsitsiklis88, Tsitsiklis93} identified the optimal decision
rules for individual nodes and considered asymptotic error rates in terms of
the number of bits. These results were recently extended to the nonasymptotic regime in Pensia,
Jog, and Loh~\cite{PenJL22, PenJL22-isit}.

\emph{Privacy-constrained hypothesis testing} has been studied in the asymptotic and nonasymptotic regimes under different notions of privacy.
The local privacy setting, which is relevant to this paper, is similar to the
decentralized detection model in Tsitsiklis~\cite{Tsitsiklis93}, except that
the each node's communication to the central server is private.
This is achieved by passing observations through private channels.
Liao, Sankar, Calmon, and Tan~\cite{LiaSCT17a, LiaSCT17b} considered maximizing
the error exponent under local privacy notions defined via maximal leakage and
mutual information.
Sheffet~\cite{Sheffet2018} analyzed the performance of the randomized response
method for LDP for hypothesis testing.
Gopi, Kamath, Kulkarni, Nikolov, Wu, and Zhang~\cite{GopKKNWZ20} showed that
$M$-ary hypothesis testing under pure LDP constraints requires exponentially
more samples ($\Omega(M)$ instead of $O(\log M)$).
Closely related to the instance-optimal algorithms in our paper, Kairouz, Oh,
and Viswanath~\cite{KaiOV16} presented an algorithm to find LDP channels that
maximize the output divergence for two fixed probability distributions at the
channel input; the proposed algorithm runs in time exponential in the domain
size of the input distributions.\footnote{We remark, however, that the algorithm in Kairouz, Oh, and Viswanath~\cite{KaiOV16} is applicable to a wider class of objective functions, which they term ``sublinear.''}
Note that divergences are directly related to error exponents and sample
complexities in binary hypothesis testing.
The results of Kairouz, Oh,
and Viswanath~\cite{KaiOV16} on extreme points of the polytope of LDP channels were strengthened in Holohan, Leith, and Mason~\cite{HolLM17}, which characterized the extreme points in special cases.
We were able to find only two
other papers that consider instance optimality, but in rather special
settings~\cite{GhaGKMZ21,AsiFT22}.
For simple binary hypothesis testing in the \emph{global} differential privacy
setting, Canonne, Kamath, McMillan, Smith, and Ullman~\cite{CanKMSU19}
identified the optimal test and corresponding sample complexity.
Bun, Kamath, Steinke, and Wu~\cite{BunKSW19} showed that $O(\log M)$ samples
are enough for $M$-ary hypothesis testing in the global differential privacy
setting.

\subsection{Organization}
This paper is organized as follows: \Cref{sec:preliminaries_and_facts} records standard results.
	\Cref{sec:hypothesis_testing} focuses on the sample complexity of hypothesis testing under privacy constraints. \Cref{sec:ext-points-comm} considers extreme points of the joint range under
	communication constraints.
	\Cref{sec:ext-points-priv-comm} characterizes the extreme points under both privacy and communication constraints.
	\Cref{sec:extensions_to_other_privacy} explores other notions of
	privacy beyond pure LDP.
	Finally, we conclude with a discussion in \Cref{sec:conclusion}.
	We defer proofs of some intermediate results to the appendices.

\section{Preliminaries and Facts}
\label{sec:preliminaries_and_facts}

\paragraph{Notation:}

Throughout this paper, we will focus on discrete distributions.
For a natural number $k \in \N$,
we use $[k]$ to denote the set $\{1,\dots,k\}$ and $\Delta_k$ to denote the set of distributions over $[k]$.
We represent a probability distribution $p \in \Delta_k$ as a vector in $\R^k$. 
Thus, $p_i$ denotes the probability of element $i$ under $p$.
Given two distributions $p$ and $q$, let $\dtv(p,q):= \frac{1}{2} \sum_i |p_i - q_i|$ and $\hel^2(p,q) := \sum_{i}(\sqrt{p_i} - \sqrt{q_i})^2$ denote the total variation distance and Hellinger divergence between $p$ and $q$, respectively.

We denote channels with bold letters such as $\bT$.
As the channels between discrete distributions can be represented by rectangular column-stochastic matrices (each column is nonnegative and sums to one), we also use bold capital letters, such as $\bT$, to denote the corresponding matrices.
In particular, if a channel $\bT$ is from $[k]$ to $[\ell]$,
we denote it by an $\ell \times k$ matrix, where each of the $k$ columns is in $\Delta_\ell$. In the same vein, for a column index $c \in [k]$ and a row index $r \in [\ell]$, we use $\bT(r,c)$  to refer to the entry at the corresponding location. 
 For a channel $\bT : \cX \to \cY$ and a distribution $p$ over $\cX$, we use $\bT p$ to denote the distribution over $\cY$ when $X \sim p$ passes through the channel $\bT$.
In the notation above, when $p$ is a distribution over $[k]$, represented as a vector in $\R^k$, and $\bT$ is a channel from $[k] \to [\ell]$, represented as a matrix $\bT \in \R^{\ell \times k}$, the output distribution $\bT p$ corresponds to the usual matrix-vector product. We shall also use $\bT$ to denote the stochastic map transforming the channel input $X$ to the channel output $Y = \bT(X)$.
Similarly, for two channels $\bT_1$ and $\bT_2$ from $[k_1]$ to $[k_2]$ and $[k_2]$ to $[k_3]$, respectively, the channel $\bT_3$ from $[k_1]$ to $[k_3]$ that corresponds to applying $\bT_2$ to the output of $\bT_1$ is equal to the matrix product $\bT_2 \times \bT_1$. 

Let $\cT_{\ell,k}$ be the set of all channels that map from $[k]$ to $[\ell]$.
We use $\cTT_{\ell,k}$ to denote the subset of $\cT_{\ell,k}$ that corresponds to threshold channels (cf.\ \Cref{def:thresh-channel}). 
We use $\cP^\epsilon_{\ell,k}$ to denote the set of all $\epsilon$-LDP channels from $[k]$ to $[\ell]$.
Recall that for two distributions $p$ and $q$, we use $\nstar(p,q, \epsilon)$ (respectively, $\nstar(p,q, \epsilon, \ell)$) to denote the sample complexity of simple binary hypothesis testing under privacy constraints (respectively, both privacy and communication constraints).

For a set $A$, we use $\conv(A)$ to denote the convex hull of $A$. For a convex set $A$, we use $\ext(A)$ to denote the set of extreme points of $A$.
Finally, we use the following notations for simplicity: (i) $\lesssim$, $\gtrsim$, and $\asymp$ to hide positive constants, and (ii) the standard asymptotic notation $O(\cdot)$, $\Omega(\cdot)$, and  $\Theta (\cdot)$.
Finally, we use $\widetilde{O} (\cdot)$, $\widetilde{\Omega}(\cdot)$, and $\widetilde{\Theta}$ to hide poly-logarithmic factors in their arguments.

\subsection{Convexity}

We refer the reader to Bertsimas and Tsitsiklis~\cite{BerTsi97} for further details.
We will use the following facts repeatedly in the paper, often without mentioning them explicitly:
\begin{fact}[Extreme points of linear transformations]
	\label{fact:ext-point-linear-transform}
	Let $\cA$ be a convex, compact set in a finite-dimensional space.
	Let $\bT$ be a linear function on $\cA$, and define the set $\cA':= \{\bT x: x
		\in \cA \}$.
	Then $\cA'$ is convex and compact, and $\ext(\cA') \subseteq \{\bT x: x \in
		\ext(\cA)\}$.
\end{fact}
\begin{fact}
	\label{fact:conv-ext}
	Let $\cA$ be a convex, compact set.
	If $\cA = \conv(\cB)$ for some set $\cB$, then $\ext(\cA) \subseteq \cB$.
\end{fact}
\begin{fact}[Number of vertices and vertex enumeration]
	\label{fact:vertex-enumeration}
	Let $\cA \subseteq \R^n$ be a bounded polytope defined by $m$ linear
	inequalities.
	The number of vertices of $\cA$ is at most ${m \choose n}$.
	Moreover, there is an algorithm that takes $e^{ O(n \log m)}$
	time and output all the vertices of $\cA$.\footnote{Throughout this paper, we
		assume the bit-complexity of linear inequalities is bounded.
	}
\end{fact}

\begin{fact}[Extreme points of channels]
	\label{fact:ext-point-channels}
	The set of extreme points of $\cT_{\ell,k}$ is the set of all deterministic channels from
	$[k]$ to $[\ell]$.
\end{fact}

\subsection{Local Privacy}

We state standard facts from the privacy literature here.

\begin{definition}[Randomized response]
	\label{def:rand-response}
	For an integer $k \geq 2$, the $k$-ary \emph{randomized response channel with
	privacy parameter $\epsilon$} is a channel from $[k]$ to $[k]$ defined as
	follows: for any $i \in [k]$, $\bT(i)=i$ with probability $\frac{e^\epsilon}{
			(k-1) + e^\epsilon}$ and $\bT(i) = j$ with probability $\frac{1}{(k-1) +
			e^\epsilon}$, for any $j \in [k]\setminus \{i\}$.
	The standard randomized response \cite{War65} corresponds to $k = 2$, which we
	denote by $\RR^\epsilon$. We omit $\epsilon$ in the superscript when it is clear
	from
	context.
\end{definition}
We will also use the following result on the extreme points for \Cref{lem:worst-case-samp-comp}.
\begin{fact}[{Extreme points of the LDP polytope in special cases~\cite{HolLM17}}]
	\label{fact:extreme-pt-special-case}
	We mention all the extreme points of $\cP^\epsilon_{\ell,k}$ (up to permutation
	of rows and columns; if a channel is an extreme point, then any permutation of
	rows and/or columns is an extreme point) below for some special cases.
	\begin{enumerate}
		\item (Trivial extreme points) A channel with one row of all ones and the rest of the rows with zero values is always an extreme point of $\cP^{\epsilon}_{\ell,k}$.
		      We call such extreme points \emph{trivial}.
		\item ($\ell=2$ and $k\geq 2$) All non-trivial extreme points of $\cP^\epsilon_{2,k}$ are of the form (up to permutation of rows):
		      \begin{align*}
	\left[
	 \begin{matrix}
	 a & a &  \cdots & a & 1-a & 1 -a &\cdots & 1-a \\
	 1- a & 1- a &  \cdots & 1-a & a & a &\cdots& a
	 \end{matrix}
	 \right],
	 \end{align*}
		      where $a/(1-a) = e^\epsilon$.
		      In other words, the columns are of only two types, containing $a$ and $1-a$.
	\end{enumerate}

\end{fact}

\subsection{Hypothesis Testing}

In this section, we state some standard facts regarding hypothesis testing and
divergences that will be used repeatedly.

\begin{fact}[Hypothesis testing and divergences; see, for example, Tsybakov~\cite{Tsybakov09}]
	Let $p$ and $q$ be two arbitrary distributions. 
	Then:
	\label{fact:sample-complexity}
	\label{fact:sample-complexity-eps-small}
	\label{fact:tv-hel}
	\label{fact:output-sze-same}
	\label{fact:id-channel}
	\begin{enumerate}
		\item We have $\dtv^2(p,q) \leq \hel^2(p,q) \leq 2\dtv(p,q)$.
		\item (Sample complexity of non-private hypothesis testing) We have $\nstar(p,q) \asymp \frac{1}{\hel^2(p,q)}$. 
		\item (Sample complexity in the high-privacy regime) For every $\epsilon\leq 1$, we have $\nstar(p,q,\epsilon) \asymp \frac{1}{\epsilon^2\dtv^2(p,q)}$.
		      See the references~\cite[Theorem 1]{DucJW18}, \cite[Theorem 2]{AsoZha22}, and~\cite[Theorem
			      5.1]{JosMNR19}.
		\item (Restricting the size of the output domain) Let $p$ and $q$ be distributions over $[k]$.
		      Then $\nstar(p,q,\epsilon) \asymp \nstar(p,q,\epsilon, k)$.
		      This follows by applying Theorem 2 in Kairouz, Oh, and Viswanath \cite{KaiOV16} to $\hel^2(\cdot,\cdot)$.
		\item (Choice of identical channels in \Cref{def:shtgeneric}) Let $\bT$ be a channel that maximizes $\hel^2(\bT p, \bT q)$ among all channels in $\cC$.
		      Then the sample complexity of hypothesis testing under the channel constraints
		      of $\cC$ is $\Theta\left(\frac{1}{\hel^2(\bT p, \bT q)}\right)$.
		      See Lemma 4.2 in Pensia, Jog, and Loh~\cite{PenJL22}.
	\end{enumerate}
\end{fact}

\begin{fact}[{Preservation of Hellinger distance under communication constraints (Theorem 1 in Bhatt, Nazer, Ordentlich, and Polyanskiy~\cite{BhaNOP21} and Corollary 3.4 in Pensia, Jog, and Loh~\cite{PenJL22})}]
	\label{fact:comm-constraints}
	Let $p$ and $q$ be two distributions on $[k]$.
	Then for any $\ell \in \N$,
	there exists a channel $\bT$ from $[k]$ to $[\ell]$, which can be computed in time polynomial in $k$, such that
	\begin{align}
 \label{eq:comm-const}
	\hel^2(p,q) \lesssim \hel^2(\bT p, \bT q)\cdot \left(1 + \frac{\ell}{\min(k,
				\log(1/\hel^2(p,q)))}\right).
	\end{align}
    Moreover, this bound is tight in the following sense: for every choice of $\rho \in (0,1)$, there exist two distributions $p$ and  $q$ such that $\hel^2(p,q) \asymp \rho$, and for every channel $\bT \in \cT_{\ell,k}$,  the right-hand side of inequality \Cref{eq:comm-const} is further upper-bounded by $O(\rho)$.
\end{fact}

\section{Locally Private Simple Hypothesis Testing}
\label{sec:hypothesis_testing}
In this section, we provide upper and lower bounds
for locally private simple hypothesis testing.
This section is organized as follows:
In \Cref{sec:binary-dist-samp-comp},
we derive instance-optimal bounds when both distributions are binary.
We then prove minimax-optimal bounds for general distributions (with
support size at least three)\new{:}
Lower bounds on sample complexity are proved in
\Cref{sec:worst-case-samp-comp-lwr-bds} and upper bounds in
\Cref{sec:gen-dist-samp-comp-upr-bds}.
Proofs of some of the technical arguments are deferred to the appendices.

\subsection{Binary Distributions and Instance-Optimality of Randomized Response}
\label{sec:binary-dist-samp-comp}

We first consider the special case when $p$ and $q$ are both binary
distributions.
Our main result characterizes the instance-optimal sample complexity in this
setting:
\LemBinaryRRSampleComplexity*
By \Cref{fact:sample-complexity}, the proof of \Cref{lem:binary-samp-comp}
is a consequence of the following bound on the strong data processing inequality for
randomized
responses:
\begin{restatable}[Strong data processing inequality for Hellinger
		divergence]{proposition}{sdpiRRbinary}
	\label{prop:SDPI-binary}
	Let $p$ and $q$ be two binary distributions. Then
	\begin{align*}
		\max_{\bT \in \cP^\epsilon_{2,2}}	\hel^2\left( \bT p, \bT q \right) & \asymp
		\begin{cases}
			\epsilon^2 \cdot \dtv^2(p,q), 
				& \text{if } \epsilon \leq 1\\ 
			\min\left(e^{\epsilon} \cdot \dtv^2(p,q), \hel^2(p,q) \right), 
				& \text{otherwise.}
		\end{cases}
	\end{align*}
	Moreover, the maximum is achieved by the randomized response channel.
\end{restatable}
\begin{proof}
   Let $\cA = \{(\bT p, \bT q): \bT \in \cP^\epsilon_{2,2}\}$ be the joint range of $p$ and $q$ under $\epsilon$-LDP privacy constraints.
    Since $\cA$ is a convex set and $\hel^2$ is a convex function over $\cA$,
    the maximizer of $\hel^2$ in $\cA$ is an extreme point of $\cA$.
    Since $\cA$ is a linear transformation of $\cP^\epsilon_{2,2}$,
    \Cref{fact:ext-point-linear-transform} implies that any  extreme point of
    $\cA$ is obtained by using a channel $\bT$ corresponding to an extreme point of $\cP_{2,2}^\epsilon$.
    By \Cref{fact:extreme-pt-special-case}, the only extreme point of $\cP_{2,2}^\epsilon$ is the randomized response channel $\RR^\epsilon$.
    Thus, in the rest of the proof, we consider $\bT = \RR^\epsilon$.

	By abusing notation, we will also use $p$ and $q$ to denote the probabilities of
	observing $1$ under the two respective distributions.
	Without loss of generality, we will assume that $0 \leq p \leq q$ and $p \leq
		1/2$.
	We will repeatedly use the following claim, which is proved in
	\Cref{sec:supplementary}:
	\begin{restatable}[Approximation for Hellinger divergence of binary
			distributions]{claim}{ApprxHel}
		\label{claim:Hellinger-binary-taylor}
		Let $p,q \in [0,1]$.
		Let $\Ber(p)$ and $\Ber(q)$ be the corresponding Bernoulli distributions with $
			\min(p,q) \leq 1/2$.
		Then
		\begin{align*}
			\hel^2 \left( \Ber(p), \Ber(q) \right) \asymp
			\frac{\dtv^2(\Ber(p),\Ber(q))}{\max(p,q)}.
		\end{align*}
	\end{restatable}
	Applying \Cref{claim:Hellinger-binary-taylor}, we obtain
	\begin{align}
	\label{eq:rr-binary-orign}
		\hel^2(p,q) \asymp \frac{\dtv^2(p,q)}{q}.
	\end{align}
	We know that the transformed distributions $p':=\RR^\epsilon p$ and $q':=\RR^\epsilon
		q$ are binary distributions;
	by abusing notation, let $p'$ and $q'$ also be the corresponding real-valued parameters associated with
	these binary distributions.
	By the definition of the randomized response, we have
	\begin{align}
	\label{eq:rr-binary}
		p' :=
		\frac{p(e^{\epsilon } - 1) + 1 }{1 + e^{\epsilon } }, \quad \text{and} \quad q' :=
		\frac{q(e^{\epsilon } - 1) + 1 }{1 + e^{\epsilon } }.
	\end{align}
	Consequently, we have $0 \leq p' \leq q'$ and $p' \leq 1/2$.
	We directly see that
	\begin{align*}
		\dtv(p',q') = q' - p' =
		\frac{(q-p)(e^{\epsilon} - 1)}{e^{\epsilon}+ 1} = \dtv(p,q) \cdot
		\frac{e^{\epsilon} - 1}{e^{\epsilon}+ 1}.
	\end{align*}
	We now apply \Cref{claim:Hellinger-binary-taylor} below to the
	distributions $p'$ and $q'$:
	\begin{align*}
		\hel^2(p', q') &\asymp
		\frac{\dtv^2(p',q')}{q'} \\ 
		&= \dtv^2(p,q) \cdot \left( \frac{e^{\epsilon} -
			1}{e^{\epsilon}+ 1} \right)^2 \cdot \frac{1 + e^\epsilon}{q (e^{\epsilon} -
			1)+1} && \left(\text{using equation } \Cref{eq:rr-binary}\right)\\ 
		&\asymp \dtv^2(p,q) \cdot \frac{(e^{\epsilon} - 1)^2}{e^{\epsilon}+ 1} \cdot\min
		\left( 1, \frac{1}{q(e^\epsilon -1)} \right) 
		&& \left(\text{using $\frac{1}{a+b} \asymp \min\left(\frac{1}{a}, \frac{1}{b}\right)$ for $a,b > 0$ } \right)
		\\
		 &\asymp \dtv^2(p,q)\cdot 
		\frac{(e^{\epsilon} - 1)^2}{e^{\epsilon}+ 1} \cdot \min \left( 1,
		\frac{\hel^2(p,q)}{\dtv^2(p,q)(e^\epsilon -1)} \right) && \Bigg(\text{using equation } \Cref{eq:rr-binary-orign} {\text{ and } } \\  
       &&&  \,\,\,\,{\min(1,a) \asymp \min(1,b) \text { if } a \asymp b }\Bigg) \\
&\asymp \min \left(
		\dtv^2(p,q) \cdot \frac{(e^{\epsilon} - 1)^2}{e^{\epsilon}+ 1}, \hel^2(p,q)
		\cdot \frac{e^{\epsilon} - 1}{e^{\epsilon}+ 1} \right) \\ 
		&\asymp
		\begin{cases}
			\min \left(
		\dtv^2(p,q) \cdot \epsilon^2\,, \hel^2(p,q)
		\cdot \epsilon\right) , & \text{if } \epsilon
			\leq 1,\\ 
			\min \left(
		\dtv^2(p,q) \cdot e^\epsilon\,, \hel^2(p,q) \right), &
			\text{otherwise,}
		\end{cases}
    && {\left(\text{using $e^{\epsilon} -1 \asymp \epsilon$ for $\epsilon \leq 1$ and $e^\epsilon$ otherwise} \right)}
    \\
		&\asymp
		\begin{cases}
			\epsilon^2 \dtv^2(p,q), & \text{if } \epsilon
			\leq 1, \\ \min\left( \dtv^2(p,q) e^{\epsilon}, \hel^2(p,q) \right), &
			\text{otherwise,}
		\end{cases}
	\end{align*}
	where the last step uses the inequality $\dtv^2(p,q) \leq \hel^2(p,q)$ from
	\Cref{fact:tv-hel}.
\end{proof}

\subsection{General Distributions: Lower Bounds and Higher Cost of Privacy}
\label{sec:worst-case-samp-comp-lwr-bds}

In this section, we establish lower bounds for the sample complexity of private
hypothesis testing for general
distributions.
In the subsequent section, the lower bounds will be shown to be tight up to
logarithmic factors.

We formally state the lower bound in the statement below:
\LemSampCompWorstCase*
We provide the proof below.
We refer the reader to \Cref{rem:binary-vs-worst} for further discussion on
differences between the worst-case sample complexity of general distributions and the sample complexity of binary distributions (cf.\
\Cref{lem:binary-samp-comp}).
We note that a similar construction is mentioned in Canonne, Kamath, McMillan, Smith, and Ullman~\cite[Section 1.3]{CanKMSU19}; however, their focus is on the \emph{central model} of differential privacy.

\subsubsection{Proof of \Cref{lem:worst-case-samp-comp}}
\begin{proof}
	The case when $\epsilon \leq 1$ follows from
	\Cref{fact:sample-complexity-eps-small}.
	Thus, we set $\epsilon \geq 1$ in the remainder of this section.
	We start with a helpful approximation for computing the Hellinger divergence,
	proved in \Cref{sec:supplementary}:
	\begin{restatable}[Additive approximation for $\sqrt{\cdot}$
		]{claim}{ClApproxHel}
		\label{cl:approx-sqRt}
		There exist constants $0< c_1 \leq c_2 $ such that for $0 < y \leq x $, we have
		 $
			c_1\cdot \frac{y^2}{x} \leq (\sqrt{x} - \sqrt{x - y})^2 \leq c_2 \cdot
			\frac{y^2}{x} $.
	\end{restatable}

	For some $\gamma \in (0,0.25)$ and $\delta > 0$ to be decided later,
	let $p$ and $q$ be the following ternary distributions:
	\begin{align*}
		p =
		\left[
			\begin{matrix}
				0 \\ 1/2 \\ 1/2
			\end{matrix}
			\right], \quad \text{and} \quad q =
		\left[
			\begin{matrix}
				2 \gamma^{1 + \delta} \\ 1/2 + \gamma - \gamma^{1 +
					\delta} \\ 1/2 - \gamma - \gamma^{1 + \delta}
			\end{matrix}
			\right].
	\end{align*}
    Since $\gamma \leq 0.25$ and $\delta\geq 0$, these two are valid distributions.
    
    Observe that $\dtv(p,q) = \gamma + \gamma^{1+\delta} \asymp \gamma$ and 
    and $\hel^2(p,q) \asymp \gamma^{1+ \delta}$ by \Cref{cl:approx-sqRt}.
	We choose $\gamma$ and $\delta$ such that $ \nu = \dtv(p,q)$ and $ \rho = \hel^2(p,q)$.
    Such a choice of $\gamma$ and $\delta$ can be made by the argument given in \Cref{sec:valid-choice-worst-case} as long as $\nu \in (0,0.5)$ and $\rho \in [2 \nu^2, \nu]$.
  Thus, these two distributions satisfy the first two conditions of the theorem
	statement.

    In the rest of the proof, we will use the facts that $\gamma^{1 + \delta} \asymp \rho$ and $\gamma \asymp \nu$.
    In particular, we have $\gamma^2 \lesssim \gamma^{1+ \delta} \lesssim \gamma$.

	Since both $p$ and $q$ are supported on [3], we can restrict our attention to
	ternary output channels (see~\Cref{fact:id-channel}). In fact, we can go a step further and restrict our attention to only binary output channels. This is because, suppose we have a ternary output channel $\bT$ with corresponding output distributions being $(p', q'):= (\bT p, \bT q)$. There exists some $i \in [3]$ such that $\left(\sqrt{ p'(i)}-\sqrt{q'(i)}\right)^2 \ge \hel^2(p', q')/3$. Consider a channel $\bT'$ that maps $[3]$ to [2] deterministically, by sending $i$ to 1 and $[3] \setminus \{i\}$ to 2. Set $(p'', q'') := (\bT' p', \bT' q')$. We note two facts: (i)  $\hel^2(p'', q'') \asymp \hel^2(p', q')$, and (ii) $\bT' \bT$ is an $\epsilon$-private channel from [3] to [2] because of the composition property of private channels. In other words, for every $\epsilon$-private channel from [3] to [3], there exists an $\epsilon$-private channel from [3] to [2] that has the same output Hellinger distance (up to multiplicative constants). Thus, 
$$\max_{\bT \in \cP^\epsilon_{3,3}} \hel^2(\bT p, \bT q)\asymp \max_{\bT \in \cP^\epsilon_{3,2}} \hel^2(\bT p, \bT q).$$
	We will establish the following result: for all $\epsilon $ such that $e \leq
		e^\epsilon \lesssim
		\frac{1}{\hel^2(p,q)}$, we have
	\begin{align}
	\label{eq:desired-h2-ternary}
		\max_{\bT \in \cP^\epsilon_{3,2}} \hel^2(\bT p, \bT q) \asymp
		\max\left( \dtv^2(p,q), \hel^4(p,q) e^\epsilon \right) \asymp
		\max( \gamma^2, e^\epsilon\gamma^{2 + 2 \delta}).
	\end{align}
	By \Cref{fact:sample-complexity}, equation~\Cref{eq:desired-h2-ternary} implies that for
	$e \leq e^\epsilon \lesssim \frac{1}{\hel^2(p,q)}$, we have
	\begin{align}
	\label{eq:desired-samp-ternary}
		\nstar(p,q,\epsilon) \asymp
		\min\left( \frac{1}{\dtv^2(p,q)}, \frac{1}{e^\epsilon \cdot \hel^4(p,q)}\right).
	\end{align}
	Let $\epsilon_0$ be the right endpoint of the range for $\epsilon$ above,
	i.e., $e^{\epsilon_0} \asymp \frac{1}{\hel^2(p,q)}$.
	Then equation~\Cref{eq:desired-samp-ternary} shows that $\nstar(p,q, \epsilon_0) \asymp
		1/\hel^2(p,q) \asymp \nstar(p,q)$.
	Since for any $\epsilon$ such that $\epsilon > \epsilon_0$, we have
	$\nstar(p,q,\epsilon) \in \left[ \nstar(p,q,\epsilon_0), \nstar(p,q)\right]$,
	the desired conclusion in equation~\Cref{eq:worst-case-samp-comp} holds for $\epsilon > \epsilon_0$, as well.
	Thus, in the remainder of this proof, we will focus on establishing equation~\Cref{eq:desired-h2-ternary}.

	Since $\hel^2(\cdot,\cdot)$ is a convex, bounded function and the set of
	$\epsilon$-LDP channels is a convex polytope, it suffices to restrict our
	attention only to the extreme points of the polytope.
	As mentioned in \Cref{fact:extreme-pt-special-case}, these extreme points are
	of two types:
	\begin{enumerate}[wide, label = \textbf{Case \Roman*}.
		]
		\item (Trivial extreme point) Any such extreme point $\bT$ maps the entire domain to a single point with
		      probability $1$.
		      After transformation under this channel, all distributions become
		      indistinguishable, giving $\hel(\bT p, \bT q) = 0$.
		\item (Deterministic binary channel cascaded with the randomized response channel)
		      This corresponds to the case when $\bT = \RR^\epsilon
			      \times \bT'$, where $\bT'$ is a deterministic threshold
		      channel from $[3]$ to $[2]$.\new{\footnote{We use \Cref{thm:ext-point-priv-comm} to restrict our attention only to threshold channels.}}
		      There are two non-trivial options for choosing $\bT'$, which we analyze below.

		      The first choice of $\bT'$ maps $\{1\}$ and $\{2,3\}$ to different elements.
		      The transformed distributions $p'$ and $q'$ are $[0,1]$ and $[2 \gamma^{1 +
						      \delta}, 1 - 2 \gamma^{1 + \delta}]$, respectively.
		      Using \Cref{cl:approx-sqRt}, we obtain $\hel^2(p',q') \asymp \gamma^{1 +
				      \delta}$ and $ \dtv(p',q')
			      \asymp \gamma^{1+ \delta}$.
		      Let $p''$ and $q''$ be the corresponding distributions after applying the
		      randomized response with parameter $\epsilon$.
		      Since $p'$ and $q'$ are binary distributions, we can apply
		      \Cref{prop:SDPI-binary} to obtain
		      \begin{equation*}
		      \hel^2(p'',q'') \asymp \min(\hel^2(p',q'),
			      e^{\epsilon} \dtv^2(p',q')) \asymp \min (\gamma^{1+ \delta}, e^{\epsilon}
			      \gamma^{2+ 2\delta}) \asymp \gamma^{1+ \delta}\cdot \min (1 , e^{\epsilon}
			      \gamma^{1
				      + \delta}),
		      \end{equation*}
		      which is equal to $e^{\epsilon} \cdot \gamma^{2 + 2\delta} $ in
		      the
		      regime of interest and consistent with the desired expression in
		      equation~\Cref{eq:desired-h2-ternary}.

		      The second choice of $\bT'$ maps $\{1,2\}$ and $\{3\}$ to different elements.
		      The transformed distributions $p'$ and $q'$ are $[1/2,1/2]$ and $[ 1/2 + \gamma
					      + \gamma^{1 + \delta}, 1/2 - \gamma - \gamma^{1 + \delta}]$, respectively.
		      Applying \Cref{cl:approx-sqRt}, we observe that $\hel^2(p',q') \asymp \gamma^2$
		      and $ \dtv(p',q') \asymp \gamma$.
		      Let $p''$ and $q''$ be the corresponding distributions after applying the
		      randomized response with parameter $\epsilon$.
		      Applying \Cref{prop:SDPI-binary}, we
		      obtain
		      \begin{equation*}
		      \hel^2(p'',q'') \asymp \min(\hel^2(p',q'),
			      e^\epsilon \dtv^2(p',q')) \asymp \min (\gamma^2, e^\epsilon \gamma^{2}) \asymp
			      \gamma^2
		      \end{equation*}
		      in the regime of interest.
		      Again, this is consistent with equation~\Cref{eq:desired-h2-ternary}.

	\end{enumerate}

	Combining the cases above, the maximum Hellinger divergence after applying
	any $\epsilon$-LDP channel is $ \Theta(\gamma^2 \cdot \max(1 , e^\epsilon
			\gamma^{2\delta}))$, as desired.
\end{proof}

\subsection{General Distributions: Upper Bounds and Minimax Optimality}
\label{sec:gen-dist-samp-comp-upr-bds}
We now demonstrate an algorithm that finds a private channel matching the 
minimax rate in \Cref{lem:worst-case-samp-comp} up to logarithmic factors.
Moreover, the proposed algorithm is both computationally efficient and
communication efficient.

\LemUppBoundCommEff*
In comparison with \Cref{lem:worst-case-samp-comp}, we see that the test above is
minimax optimal up to logarithmic factors over the class of
distributions with fixed Hellinger divergence and total variation distance.
The channel $\bT$ satisfying this rate is of the following simple form: a deterministic binary channel $\bT'$, followed by the randomized response. 
In fact, we can take $\bT'$ to be either Scheffe's test (which preserves the total variation distance) or the binary channel from \Cref{fact:comm-constraints} (which preserves the Hellinger divergence), whichever of the two is better. 
We provide the complete proof in
\Cref{sec:prf-worst-case-samp-comp-lwr-bds}.

One obvious shortcoming of \Cref{thm:samp-comp-uprbds} is that even when
$\epsilon \to \infty$,
the test does not recover the optimal sample
complexity of $1/\hel^2(p,q)$, due to the logarithmic multiplier $\alpha$.
We now consider the case when $e^\epsilon \gtrsim \frac{1}{\hel^2(p,q)}$ and exhibit
a
channel that achieves the optimal sample complexity as soon as $e^\epsilon
	\gtrsim \frac{1}{\hel^2(p,q)} \log\left(\frac{1}{\hel^2(p,q)}\right)$.
Thus, privacy can be attained essentially for free in this regime.

\LemFreePrivacy*

We note that the size of the output domain of $\lceil \log(\nstar(p,q))\rceil$ is tight in the sense that any channel that achieves the sample complexity within constant factors of $\nstar(p,q)$ must use an output domain of size at least $\Omega(\log(\nstar(p,q)))$;
this follows by the tightness of \Cref{fact:comm-constraints} in the worst case.
Consequently, the channel $\bT$ achieving the rate above is roughly of the form (1) a communication-efficient channel from \Cref{fact:comm-constraints} that preserves the Hellinger divergence up to constant factors, followed by (2) an $\ell$-ary randomized response channel, for $\ell\geq 2$.

We give the proof of \Cref{lem:privacy-free} in \Cref{sec:proof_lem_privacy_free} and defer the proofs of some of the intermediate results to \Cref{app:rand-resp-low-priv}.
\subsubsection{Proof of \Cref{thm:samp-comp-uprbds}}
\label{sec:prf-worst-case-samp-comp-lwr-bds}
In this section, we provide the proof of \Cref{thm:samp-comp-uprbds}.
We first note that this result can be slightly
strengthened, replacing $\alpha$ by $\frac{\nstar_{ \text{binary}}}{\nstar}$,
where $\nstar_{ \text{binary}}$ is the sample complexity of hypothesis testing
under binary communication constraints.
This choice of $\alpha$ is smaller, by Pensia, Jog, and Loh~\cite[Corollary 3.4 and Theorem 4.1]{PenJL22}.

\begin{proof}
	The case of $\epsilon \leq 1$ follows from
	\Cref{fact:sample-complexity-eps-small}; thus, we focus on the setting where
	$\epsilon > 1$.

	We will establish these bounds via \Cref{prop:SDPI-binary}, by using a binary deterministic channel,
	followed by the binary
	randomized response channel. A sample complexity of $1/\dtv^2(p,q)$ is direct by using the channel for
	$\epsilon = 1$.
	Thus, our focus will be on the term $\frac{1}{
			e^\epsilon \hel^4(p,q)}$.

	Let $\bT' \in \cT_{2,k}$ be a deterministic binary output channel to be decided
	later.
	Consider the channel $\bT = \RR^\epsilon \times \bT'$.
	By \Cref{prop:SDPI-binary},
	we have
	\begin{align*}
	\hel^2\left(\RR^\epsilon \times \bT' p, \RR^\epsilon \times \bT' q  \right)
	&\asymp \min\left( e^\epsilon \dtv^2\left( \bT'p,\bT'q \right), \hel^2\left( \bT'p, \bT'q \right)  \right) \\
	&\geq \min\left( e^\epsilon \hel^4\left( \bT' p, \bT' q \right), \hel^2\left( \bT'p, \bT'q \right)  \right) \\
	&= \hel^2\left(\bT' p, \bT' q  \right)\cdot \min\left( e^\epsilon \hel^2\left( \bT' p, \bT'q \right), 1\right)
	\numberthis \label{eq:minmax-upper-bd},
	\end{align*}
	where the first inequality uses \Cref{fact:tv-hel}

	If we choose the channel $\bT'$ from \Cref{fact:comm-constraints}, we have
 $\hel^2(\bT 'p , \bT' q) \geq \frac{1}{\alpha} \cdot \hel^2\left(p, q
		\right)$.
	Applying this to inequality~\Cref{eq:minmax-upper-bd},
	we obtain
	\begin{align*}
	 \hel^2\left(\RR^\epsilon \times \bT' p, \RR^\epsilon \times \bT' q  \right) 
	 	&\gtrsim \frac{1}{\alpha}\cdot \hel^2\left(p, q  \right)\cdot \min\left( \frac{1}{\alpha}\cdot e^\epsilon \hel^2\left(  p, q \right), 1\right).
	 \end{align*}
	By \Cref{fact:id-channel}, the sample complexity of $\RR^\epsilon \times \bT'$, which is
	$\epsilon$-LDP, is at most
	$\frac{\alpha}{\hel^2(p,q)}\cdot \max\left( 1, \frac{\alpha}{ e^\epsilon
				\hel^2(p,q)} \right)$, which is equivalent to the desired statement.
    
Finally, the claim on the runtime is immediate, since the channel $\bT'$ from \Cref{fact:comm-constraints} can be found efficiently.
\end{proof}

\subsubsection{Proof of \Cref{lem:privacy-free}}
\label{sec:proof_lem_privacy_free}
We will prove a slightly generalized version of \Cref{lem:privacy-free} below that
works for a wider range of $\epsilon$:
\begin{proposition}
	\label{lem:k-ary-rand-resp-low-priv}
	Let $ {p}$ and $ q$ be two distributions on $[k]$ and $\epsilon > 1$.
	Then there exists an $\epsilon$-LDP channel $\bT$ from $[k]$ to $[\ell]$, for
	$\ell = \min(\lceil \log(1/\hel^2(p,q)) \rceil, k)$, such that
		\begin{align*}
		\hel^2(\bT  p, \bT  q) \gtrsim \hel^2(p,q)
		\cdot \min \left( 1, \frac{e^\epsilon \cdot \hel^2( p,  q)}{\log(1/\hel^2( p,  q))}\right) \cdot \min \left(1,
		\frac{e^\epsilon}{\log(1/\hel^2( p,  q))} \right).
	\end{align*}
Furthermore, the channel $\bT$ can be be computed in $\poly(k,\ell)$ time.
\end{proposition}
By \Cref{fact:sample-complexity}, \Cref{lem:k-ary-rand-resp-low-priv}
implies the following:
\begin{align*}
		\nstar(p,q,\epsilon) \lesssim	\nstar \cdot \max
		\left( 1, \frac{\nstar \log (\nstar)}{e^\epsilon}\right) \cdot \max \left(1,
		\frac{\log \nstar}{e^\epsilon}\right),
	\end{align*}
where $\nstar := \nstar(p,q)$.
Setting $e^\epsilon$ equal to $\nstar \log(\nstar)$ proves
\Cref{lem:privacy-free}.
Thus, we will focus on proving \cref{lem:k-ary-rand-resp-low-priv} in the rest
of this section.
We establish this result with the help of the following observations:
\begin{itemize}
	\item (\Cref{lem:k-ar-rr-general}) First, we show that the randomized response
	      preserves the contribution to the Hellinger divergence by ``comparable elements'' (elements
	      whose likelihood ratio is in the interval $\left[\frac{1}{2},2\right]$) when $\epsilon$ is large
	      compared to the support size.
	      In particular, we first define the following sets:
	      \begin{align}
	\label{eq:Adefinition}
	A &= \left\{ i \in [k]: \frac{p_i}{q_i} \in \left[\frac{1}{2},1\right) \right\} \,\, \text{and}\,\, A' = \left\{ i \in [k]: \frac{p_i}{q_i} \in [1,2] \right\}. 
\end{align}
Let $\RR^{\epsilon,\ell }$ denote the randomized response channel from $[\ell]$
	      to $[\ell]$ with privacy parameter $\epsilon$ (cf.\ \Cref{def:rand-response}). The following result is proved in
	      \Cref{app:effect_of_randomized_response_on_light_elements}:
	      	      \begin{restatable}[Randomized response preserves contribution of comparable elements]{lemma}{LemkRRGeneral}
		      \label{lem:k-ar-rr-general}
		      Let $p$ and $q$ be two distributions on $[\ell]$.
		      Suppose $\sum_{i \in
				      A \bigcup A'} (\sqrt{q_i} - \sqrt{p_i})^2 \geq \tau $.
		      Then $ \RR^{\epsilon,\ell }$,
		      for $\ell \leq e^{\epsilon}$,
		      satisfies
		      \begin{align*}
			\hel^2(
			\RR^{\epsilon,\ell } p, \RR^{\epsilon,\ell } q) \gtrsim\min\left(1, e^{\epsilon} \frac{\tau}{\ell }  \right) \cdot \tau\,.
		\end{align*}
		      Thus, when $e^{\epsilon} \gtrsim \frac{\ell}{\tau}$,
		      the randomized response preserves the original contribution of comparable elements.
	      \end{restatable}

	\item (\Cref{lem:reduction}) We then show in \Cref{lem:reduction}, proved in \Cref{app:proof_of_cref_lem_reduction}, that either we can reduce the problem to the previous special case (small support size and main contribution to Hellinger divergence is from comparable elements) or to the case when the distributions are binary (where \Cref{prop:SDPI-binary} is applicable and is, in a sense, the \emph{easy case} for privacy).
	      \begin{restatable}[Reduction to base case]{lemma}{PropReduction}
		      \label{lem:reduction}
		      Let $p$ and $q$ be two
		      distributions on $[k]$.
		      Then there is a channel $\bT$, which can be computed in  time polynomial in $k$,  that maps $[k]$ to $[\ell]$ (for $\ell $ to be
		      decided below) such that for $p' = \bT
			      p$ and $q' = \bT q$, at least one of the following holds:
		      \begin{enumerate}
			      \item For any $\ell > 2$ and $\ell \leq
				            \min \left( k, 1 + \log\left(1/\hel^2(p,q)\right)\right)$, we have
			            \begin{align*}
		\sum_{i \in B \bigcup B'} \left(\sqrt{q_i'} - \sqrt{p_i'}\right)^2 \gtrsim \hel^2(p,q) \cdot \frac{\ell }{\min\left( k, 1 + \log\left(1/\hel^2(p,q)\right)\right)},
		\end{align*}
			            where $B$ and $B'$ are defined analogously to $A$ and $A'$ in
			            equation~\Cref{eq:Adefinition}, but with respect to
			            distributions $p'$ and $q'$.

			      \item $\ell  = 2$ and $\hel^2(p',q') \gtrsim \hel^2(p,q)$.
		      \end{enumerate}

       \end{restatable}
\end{itemize}

We now provide the proof of \Cref{lem:k-ary-rand-resp-low-priv}, with the
help of \Cref{lem:k-ar-rr-general,lem:reduction}.

\begin{proof}
	(Proof of \Cref{lem:k-ary-rand-resp-low-priv})
	The channel $\bT$
	will be of the form $\bT = \RR^{\epsilon,\ell } \times \bT_1$, where
	$\bT_1$ is a channel from $[k] $ to $[\ell]$ and $\ell $ is to be decided.
	The privacy of $\bT$ is clear from the construction.

	We begin by applying \Cref{lem:reduction}.
	Let $\bT_1$ be the channel from \Cref{lem:reduction} that maps from $[k]$ to
	$[\ell]$.
	Let $p' = \bT_1 p$ and $q' = \bT_1 q$, and define $\tilde{p} = \RR^{\epsilon, \ell } p'$ and $\tilde{q} =
		\RR^{\epsilon, \ell } q'$.
The claim on runtime thus follows from \Cref{lem:reduction}.

	Suppose for now that $\bT_1$ from \Cref{lem:reduction} is a binary channel.
	Then we know that $\hel^2(p',q') \gtrsim \hel^2(p,q)$ and $\dtv(p',q') \gtrsim
		\hel^2(p',q')$, where the latter holds by \Cref{fact:tv-hel}.
	Applying \Cref{prop:SDPI-binary},
	we have
	\begin{align*}
	\hel^2(\tilde{p}, \tilde{q}) &\gtrsim \min\left( \hel^2(p',q'), e^ \epsilon \dtv^2(p',q' )\right)\\
			& \gtrsim \min\left( \hel^2(p',q'), e^ \epsilon \hel^4(p',q' )\right) \\			& \gtrsim \hel^2(p,q) \min\left( 1 , e^ \epsilon \hel^2(p,q )\right),
	\end{align*}
	which concludes the proof in this case.

	We now consider the case when $\ell > 2$ in the guarantee of
	\Cref{lem:reduction}.
	Then the comparable elements of $p'$ and
	$q'$ preserve a significant fraction of the Hellinger divergence (depending on the
	chosen value of $\ell$) between $p$ and $q$.
	Let $k' = \min \left( k, 1 + \log(1/\hel^2(p,q)) \right)$, and choose $\ell$ to
	be $\min(k', e^{\epsilon})$.
	Then \Cref{lem:reduction} implies that the contribution to the Hellinger divergence
	from
	comparable elements of $p'$ and $q'$ is at least $\tau$, for $\tau \asymp
		\hel^2(p,q) \frac{\ell}{k'}$.
	We will now apply \Cref{lem:k-ar-rr-general} to $p'$
	and $q'$ with the above choice of $\tau$.
	Since $\ell \leq e^\epsilon$ by construction, applying
	\Cref{lem:k-ar-rr-general} to $p'$ and $q'$, we obtain
	\begin{align*}
	\hel^2( \tilde{p}, \tilde{q}) 
	& \gtrsim	\tau \cdot \min \left( 1, \frac{e^\epsilon \tau}{\ell } \right)\\
	& \gtrsim \hel^2(p, q) \frac{\ell}{k'} \cdot \min \left( 1,   \frac{e^\epsilon
		\cdot \hel^2(  p, q)}{\log(1/\hel^2( p, q))}\right) \\ 
	& \gtrsim \hel^2(p, q) \cdot \min \left( 1, \frac{e^\epsilon}{\log(1/\hel^2(p,q))} \right) \cdot \min \left( 1,   \frac{e^\epsilon
		\cdot \hel^2(  p, q)}{\log(1/\hel^2( p, q))}\right)\,, 
	\end{align*}
	where the last step uses the facts that $\ell = \min(e^\epsilon, k')$ and $k' \gtrsim
		\log(1/\hel^2(p,q))$.
	This completes the proof.
\end{proof}

\section{Extreme Points of Joint Range Under Communication
  Constraints}
\label{sec:ext-points-comm}

In this section, our goal is to understand the extreme points of the set $\cA
	:= \{(\bT p, \bT q): \bT \in \cT_{\ell,k}\}$.
This will allow us to identify the structure of
optimizers of quasi-convex functions over $\cA$.
The main result of this section is the following:
\ThmJointRangeComm*

We provide the proof of \Cref{thm:ext-point-comm} in
\Cref{sec:proof-ext-point-comm}.
Before proving \Cref{thm:ext-point-comm},
we discuss some consequences for optimizing quasi-convex functions over $\cA$.
The following result proves \Cref{cor:alg-priv-quasi-cvx} for $\cC = \cT_{\ell,k}$:

\begin{corollary}[Threshold channels maximize quasi-convex functions]
	\label{cor:ext-point-comm-quasicvx}
	Let $p$ and $q$ be two distributions on $[k]$.
	Let $\cA := \{(\bT p, \bT q): \bT \in \cT_{\ell,k}\}$.
	Let $g$ be a real-valued quasi-convex function over $\cA$.
	Then
	\begin{align*}
		\max_{\bT \in \cT_{\ell,k}} g(\bT p, \bT q) = \max_{\bT \in
			\cTT_{\ell,k}} g(\bT p, \bT q).
	\end{align*}
	Moreover, the above optimization problem can be solved in time $
		\poly(k^\ell)$.\footnote{Recall that $g$ is assumed to be permutation invariant. If not, an extra factor of $\ell!$ will appear in the time complexity.}
\end{corollary}
\begin{proof}
	Observe that $\cA$ is a closed polytope.
	Let $\cX$ be the set of extreme points of $\cA$.
	Observe that $\cX \subseteq \{(\bT p, \bT q): \bT \in \cT_{\ell,k} \text{ and }
		\bT
		\text{ is deterministic}\}$, and thus is finite.
	Since $\cA$ is a closed polytope, $\cA$ is convex hull of $\cX$.
	Furthermore, the maximum of $g$ on $\cX$ is well-defined and finite, as $\cX$ is
	a
	finite set.
	Any $y \in \cA$ can be expressed as a convex combination $y = \sum_{x_i \in
			\cX} \lambda_i x_i$. Recall that an equivalent definition of quasi-convexity is that $g$ satisfies $g(\lambda x + (1-\lambda)y) \le \max (g(x), g(y))$ for all $\lambda \in [0,1]$. By repeatedly using this fact, we have
	$$g(\lambda) = g\left(\sum_{x_i \in \cX} \lambda_i x_i\right) \le \max_{x \in
			\cX} g(x).
	$$

	By \Cref{thm:ext-point-comm}, any extreme point $x \in \cX$ is obtained
	by passing $p$ and $q$ through a threshold channel.
	Thus, the maximum of $g$ over $\cX$ is attained by passing $p$ and $q$
	through a threshold channel.
	The claimed runtime is obtained by trying all possible threshold channels.
\end{proof}
\begin{remark}
	(Quasi-)convex functions of interest include all $f$-divergences,
	R\'{e}nyi divergences, Chernoff information, and $L_p$ norms.
	We note that the above result also holds for post-processing: For any fixed
	channel $\bH \in \cT_{\ell',\ell}$, we have
	\begin{align*}
		\max_{\bT \in \cT_{\ell,k}} g( \bH \bT p, \bH \bT
		q) = \max_{\bT \in \cTT_{\ell,k}} g(\bH \bT p, \bH \bT q).
	\end{align*}
	This is because $g(\bH p', \bH q')$ is a quasi-convex function of $(p', q')
		\in \cA$.
\end{remark}

\begin{remark}
If $g$ is the Hellinger divergence and $\cC = \cT_{\ell, k}$, we can conclude the following result for the communication-constrained setting: There exists a $\bT \in \cTT_{\ell, k}$ that attains the instance-optimal sample complexity (up to universal constants) for hypothesis testing under a communication constraint of size $\ell$. This result is implied in Pensia, Jog, and Loh~\cite{PenJL22} by Theorem 2.9 (which is a result from Tsitsiklis~\cite{Tsitsiklis93}) and Lemma 4.2. The above argument provides a more straightforward proof. 
\end{remark}

\subsection{Proof of
	\Cref{thm:ext-point-comm}}
\label{sec:proof-ext-point-comm}
We now provide the proof of \Cref{thm:ext-point-comm}.
\begin{proof}
	(Proof of
	\Cref{thm:ext-point-comm})
	We first make the following simplifying assumption about the likelihood ratios: there is at most a single element $i^*$ with $q_{i^*} = 0$,
	and for all other elements $i \in [k]\setminus \{i^*\}$, $p_i/q_i$ is a unique value.
	If there are two or more elements with the same likelihood ratio, we can merge
	those elements into a single alphabet without loss of generality, as we explain
	next.
	Let $p'$ and $q'$ be the distributions after merging these elements, and let $k'
		\le k$ be the new cardinality.
	Then for any channel $\bT \in \cT_{\ell,k}$, there exists another channel $\bT'
		\in \cT_{\ell,k'}$ such that $(\bT p , \bT q) = (\bT'p', \bT'q')$.
	We can then apply the following arguments to $ p'$ and $q'$.
	See \Cref{sec:regu-cond-joint-range} for more details.

	In the following, we will consider $p_i/q_i$ to be $\infty$ if $q_i = 0$, and
	we introduce the notation $\theta_i := p_i/q_i$.
	We will further assume, without loss of generality, that
	$p_i/q_i$ is strictly increasing
	in $i$.
	Since the elements are ordered with respect to the likelihood ratio, a threshold
	channel corresponds to a map that partitions the set $[k]$ into contiguous
	blocks.
	Formally, we have the following definition:
	\begin{definition}[Partitions and
			threshold partitions] We say that $\cS = (S_1,S_2,\dots,S_{\ell})$ forms an
		$\ell$-\emph{partition}
		of $[k]$ if $ \cup_{i=1}^\ell S_i = [k]$ and $S_i \cap S_j = \emptyset$ for $1
			\leq i \neq j \leq \ell$.
		We say that $\cS$ forms an $\ell$-\emph{threshold partition} of $[k]$ if in addition, for all $i < j \in \ell$, every entry of $S_i$ is less than
		every entry of $S_j$.
	\end{definition}
	As mentioned before, channels corresponding to $\ell$-threshold partitions are
	precisely the threshold channels up to a permutation of output labels.
	The channels corresponding to $\ell$-partitions are the set of all
	deterministic channels that map $[k]$ to $\ell$, which are the extreme points of $\cT_{\ell,k}$ (cf.\
	\Cref{fact:ext-point-channels}).

	Observe that $\cA$ is a convex, compact set, which is a linear transformation of
	the convex, compact set $ \cT_{\ell,k}$, and any extreme point of $\cA$ is of the form $(\bT p, \bT q)$, where $\bT$
	is an extreme point of $\cT_{\ell, k}$ (cf. \Cref{fact:ext-point-linear-transform}).
	Now suppose	$(\bT p, \bT q)$ is an extreme point of $\cA$, but $\bT$ is not
	a threshold channel.
	Thus, $\bT$ corresponds to  some $\ell$-partition $\cS$
	of
	$[k]$ that is not an $\ell$-threshold partition.
	We will now show that $(\bT p, \bT q)$ is not an extreme point of $\cA$, by
	showing that there exist two distinct channels $\bT_1 \in \cT_{\ell, k}$ and
	$\bT_2 \in \cT_{\ell, k}$
	such that the following holds:
	\begin{align}
		\label{eq:extreme-point}
		\frac{1}{2} \cdot \bT_1 p + \frac{1}{2} \cdot \bT_2 p = \bT p, \quad \text{and} \quad
		\frac{1}{2} \cdot \bT_1 q + \frac{1}{2} \cdot \bT_2 q = \bT q,
	\end{align}
	and $\bT_1 p \neq \bT p$.

	Since $\cS$ is not a $\ell$-threshold partition, there exist $1 \leq a < b < c
		\leq
		k$ and $m \neq n$ in $[\ell]$ such that $a,c \in S_m$ and $b \in S_n$, and
	$p_a/q_a
		< p_b/q_b < p_c/q_c$.
	Among $q_a, q_b$, and $q_c$, only $q_c$ is potentially zero.
	Suppose for now that $q_c \neq 0$; we will consider the alternative case shortly.

	For some $\epsilon_1 \in (0,1)$ and $\epsilon_2 \in (0,1)$ to be determined
	later, let $\bT_1$ be the following channel:
	\begin{enumerate}
		\item For $x
			      \not \in \{a, b\}$, $\bT_1$ maps $x$ to $\bT(x)$.
		\item For $x = a$ (respectively $b$), $\bT_1 $ maps $x$
		      to $m$
		      (respectively $n$) with probability $1 - \epsilon_1$
		      (respectively $1 -
			      \epsilon_2$) and to $n$ (respectively $m$) with
		      probability $\epsilon_1$
		      (respectively $\epsilon_2$).
	\end{enumerate}
	Thus, the channels $\bT$ and $\bT_1$ have all columns identical, except for
	those corresponding to inputs $a$ and $b$.
	Let $v_i$ be the $i^{\text{th}}$ column of $\bT$.
	Observe that $v_a$ is a degenerate distribution at $m \in [\ell]$ and $v_b$ is a
	degenerate distribution at $n \in [\ell]$ (equivalently, $\bT(m, a)= 1$ and
	$\bT(n, b)=1$).
	Thus, we can write
	\begin{align*}
		\bT_1 q & = \bT q + (\epsilon_2 q_b -
		\epsilon_1 q_a)  (v_a - v_b), \\ \bT_1 p & = \bT p + (\epsilon_2
		p_b - \epsilon_1 p_a)  (v_a - v_b).
	\end{align*}
	If we choose $\epsilon_1 q_a = \epsilon_2q_b$, we have $\bT_1 q = \bT
		q$ and
	\begin{align*}
		\bT_1 p & =
		\bT p + (\epsilon_2 p_b - \epsilon_1 p_a)  (v_a - v_b) \\ & = \bT
		p + (\epsilon_2 q_b \theta_b - \epsilon_1 q_a \theta_a)  (v_a -
		v_b) \\ & = \bT p + \epsilon_1 q_a (\theta_b - \theta_a)  (v_a - v_b).
		\numberthis
		\label{eq:Tpdecomposition}
	\end{align*}
	Recall that $\theta_b > \theta_a$, as mentioned above.

	We now define $\bT_2$. For some $\epsilon_3 \in (0,1)$ and $\epsilon_4
		\in (0,1)$ to be decided later, we have:
	\begin{enumerate}
		\item For $x \not \in \{b,
			      c\}$, $\bT_2$ maps $x$ to $\bT(x)$.
		\item For $x = c$ (respectively $b$), $\bT_2 $ maps $x$
		      to $m$
		      (respectively $n$) with probability $1 - \epsilon_3$
		      (respectively $1 -
			      \epsilon_4$) and to $n$ (respectively $m$) with
		      probability $\epsilon_3$
		      (respectively $\epsilon_4$).
	\end{enumerate}
	With the same arguments as before, we have
	\begin{align*}
		\bT_2 q & = \bT q +
		(\epsilon_4 q_b - \epsilon_3 q_c)  (v_c - v_b), \\ \bT_1 p & = \bT p +
		(\epsilon_4 p_b - \epsilon_3 p_c)  (v_c - v_b).
	\end{align*}
	If we choose $\epsilon_3 q_c = \epsilon_4q_b$, we have $\bT_2 q = \bT
		q$ and
	\begin{align*}
		\bT_2 p & = \bT p + (\epsilon_4 q_b \theta_b -
		\epsilon_3 q_c \theta_c)  (v_c - v_b) \\ & = \bT p + \epsilon_3 q_c
		(\theta_b - \theta_c)  (v_c - v_b) \\ & = \bT p + \epsilon_3 q_c (\theta_b
		- \theta_c)  (v_a - v_b), \numberthis \label{eq:Tpdecomposition2}
	\end{align*}
	where the last line follows by the fact that $v_a = v_c$, since
	$\bT$ maps both $a$ and $c$ to $m$ almost surely.

	Let $\epsilon_1 \in(0,1)$ and $\epsilon_3 \in (0,1)$ be such that $\epsilon _1
		q_a (\theta_b - \theta_a) = - \epsilon_3 q_c (\theta_b - \theta_c)$.
	Such a choice always exists because $\theta_b - \theta_a $ and $-(\theta_b -
		\theta_c)$ are both strictly positive and finite.
	Then equations~\Cref{eq:Tpdecomposition} and~\Cref{eq:Tpdecomposition2} imply that $\bT p =
		\frac{1}{2} \bT_1 p + \frac{1}{2} \bT_2 p$ and $\bT q = \frac{1}{2} \bT_1 q +
		\frac{1}{2} \bT_2 q$, and $\bT_1 p \neq \bT p$.
	Moreover, $\bT_1 p \neq \bT p$.
	Thus, $(\bT p , \bT q)$ is not an extreme point of $\cA$.

	We now outline how to modify the construction above when $q_c$ is zero.
	By setting $\epsilon_4 $ to be zero, we obtain $\bT_2 q = \bT q$ and
	$\bT_2 p = \bT p + (-\epsilon_3 p_c) \left( v_a - v_b \right)$.
	The desired conclusion follows by choosing $\epsilon_1$ and $\epsilon_3$ small
	enough such that
	$\epsilon_1 q_a(\theta_b - \theta_a) = - \epsilon_3 p_c$.

\end{proof}

\section{Extreme Points of Joint Range under Privacy Constraints} %
\label{sec:ext-points-priv-comm}

In the previous section, we considered the extreme points of the joint range
under communication constraints.
Such communication constraints are routinely applied in practice in the presence of
additional constraints such as local differential privacy.
However, the results of the previous section do not apply directly, as the joint
range is now a strict subset of the set in \Cref{thm:ext-point-comm}, and the
extreme points differ significantly.
For example, the threshold channels are not even private.
However, we show in this section that threshold channels still play a
fundamental role.
Our main result in this section is the following theorem:

\ThmOptStructurePriv*

Actually, our result applies to a broader family of linear programming (LP) channels
that we describe below:

\begin{restatable}[LP family of channels]{definition}{DefLPFamily}
	\label{def:family-of-channels}
	For any $\ell \in \N$,
	let $\nu = (\nu_1, \nu_2, \dots, \nu_\ell)$ and $\gamma = (\gamma_1, \gamma_2,
		\dots, \gamma_\ell)$ be two nonnegative vectors in $\R^{\ell}_+$.
	For $k \in \N$, define the set of linear programming (LP) channels $\cJ_{\ell,
			k}^{\gamma, \nu} $, a subset of $\cT_{\ell,k}$, to be the (convex) set
	of all channels from $[k]$ to $[\ell]$ that satisfy the following constraints:
	\begin{align}
\label{eq:lp-convex-set}
\text{For each row $j\in[\ell]$, and for each $i, i ' \in [k]$, we have } \bT(j,i) \leq \gamma_{j} \bT(j,i') + \nu_{j}.	
	\end{align}
\end{restatable}
When $\gamma_j = e^\epsilon$ and $\nu_j = 0$ for all $j \in [\ell]$, we recover
the set of $\epsilon$-LDP channels from $[k]$ to $[\ell]$.
Another example will be mentioned in \Cref{sec:extensions_to_other_privacy}
for a relaxed version of approximate LDP.

The rest of this section is organized as follows:
	In \Cref{sec:bound_on_number_of_unique_columns}, we show that any $\bT$ that leads to an extreme point of $\cA$ cannot have more than $2 \ell^2$ unique columns (\Cref{thm:ext-point-priv-uniq-col-poly}).
	We use this result to prove \Cref{thm:ext-point-priv-comm} in
	\Cref{sec:unique_columns_to_threshold_channels}.
	In \Cref{SecLDPCommTesting}, we apply \Cref{thm:ext-point-priv-comm} to prove
	\Cref{cor:alg-priv-quasi-cvx,cor:alg-priv-sbht}.

\subsection{Bound on the Number of Unique Columns}
\label{sec:bound_on_number_of_unique_columns}
The following result will be critical in the proof of
\Cref{thm:ext-point-priv-comm}, the main result of this section.
\begin{lemma}
	\label{thm:ext-point-priv-uniq-col-poly}
	Let $p$ and $q$ be distributions on $[k]$.
	Let $\cC$ be the set of channels from $[k]$ to $[\ell]$, from
	\Cref{def:family-of-channels}.
	Let $\cA$ be the set of all pairs of distributions that are obtained by
	applying a channel from $\cC$ to $p$ and $q$, i.e.,
	\begin{align}
		\cA = \{(\bT
		p, \bT q) \mid \bT \in \cC\}.
	\end{align}
	If $\bT$ has more than $2\ell^2$ unique columns, then $(\bT p, \bT q)$ is
	not an extreme point of $\cA$.
\end{lemma}
We prove this result in \Cref{sec:forbidden-structure} after proving a quantitatively weaker, but simpler, result in \Cref{sec:exp-bound-uniq-columns}.

\subsubsection{Warm-Up: An Exponential Bound on the Number of Unique Columns}
\label{sec:exp-bound-uniq-columns}
In this section, we first prove a weaker version of
\Cref{thm:ext-point-priv-uniq-col-poly},
where we upper-bound the number of unique columns
in the extreme points of $\cC$ from \Cref{def:family-of-channels} (not just
those that lead to extreme points of $\cA$)
by an exponential in $\ell$.
In fact, this bound will be applicable for a broader
class of channels that satisfy the following property:
\begin{condition}[Only one free entry per column]
	\label{def:one_loose_entry}
	Let $\cC$ be a convex set of channels from $[k]$ to $[\ell]$.
	Let $\bT$ be an extreme point of $\cC$.
	Then there exist numbers $\{m_1,\dots,m_\ell\}$ and $\{M_1,\dots,M_\ell\}$ such that
	for every column $c \in [k]$, there exists at most a single row $r \in [\ell]$
	such
	that $ \bT(r,c) \not \in \{m_r, M_r\}$.
	We call such entries \emph{free}.

\end{condition}
We show in \Cref{sec:properties_of_private_channels} that extreme points of the
LP channels from \Cref{def:family-of-channels} satisfy
\Cref{def:one_loose_entry}.
The following claim bounds the number of unique columns in any extreme point of
$\cC$, and thus also implies a version of \Cref{thm:ext-point-priv-comm} with
$\ell \cdot 2^{\ell -1}$ instead of $2\ell ^2$ (cf. \Cref{fact:ext-point-linear-transform}).

\begin{claim}[Number of unique columns in an extreme point is at most exponential in $\ell $]
	\label{claim:num_uniq_colums}
	Let $\cC$ be a set of channels satisfying the property of
	\cref{def:one_loose_entry}.
	Let $\bT$ be an extreme point of $\cC$.
	Then the number of unique columns in $\bT$ is at most $\ell \cdot 2^{\ell -1}$.
	In particular, $\bT$ can be written as $\bT_2 \times \bT_1$, where $\bT_1$ is a
	deterministic map from $[k]$ to $[\ell ']$ and $\bT_2$ is a map from $[\ell
				']$ to
	$[\ell]$, for $\ell ' = \ell \cdot 2^{\ell -1}$.
\end{claim}
\begin{proof}
	We use the notation from \Cref{def:one_loose_entry}.
	For each column, there are $\ell $ possible locations of a potential free
	entry.
	Let this location be $j^*$; the value at this location is still flexible.
	Now let us consider the number of ways to assign values at the remaining
	locations.
	For each $j \in [\ell] \setminus \{j^*\}$, the entry is either $m_j$ or $M_j$
	(since $j$ is not a free entry).
	Thus, there are $2^{\ell -1}$ such possible assignments.
	Since the column entries sum to one, each of those $2^{\ell -1}$ assignments
	fixes the value at the $j^*$ location, as well.
	Thus, there are at most $\ell \cdot 2^{\ell -1}$ unique columns in $\bT$.
\end{proof}

\subsubsection{Forbidden Structure in Extreme Points Using the Joint Range}
\label{sec:forbidden-structure}
In \Cref{claim:num_uniq_colums}, we considered the extreme points of LP
channels.
However, we are actually interested in a (potentially much) smaller set: the
extreme points that correspond to the extreme points of the joint range $\cA$.
In this section, we identify a necessary structural property for extreme points
of LP channels that lead to extreme points of the joint range.
We begin by defining the notion of a ``loose'' entry in a channel in $\cC$:
\begin{definition}[Loose and tight entries]
	\label{def:loose-and-tight}
	Let $\bT$ be a channel in $\cJ_{\ell,k}^{\gamma, \nu}$ from
	\Cref{def:family-of-channels} that maps from $[k]$ to $[\ell]$.
	Let $\{m_1,\dots,m_\ell\}$ and $\{M_1, \dots, M_\ell\}$ be the row-wise minimum
	and maximum entries, respectively.
	For $c \in [k]$ and $r \in [\ell]$,
	we say an entry $\bT(r,c)$ is \emph{max-tight} if
	$\bT(r,c) = M_r$ and $M_r = \gamma_r m_r + \nu_r$.
	An entry $\bT(r,c)$ is \emph{min-tight} if
	$\bT(r,c) = m_r$ and $M_r = \gamma_r m_r + \nu_r$.
	An entry that is neither max-tight nor min-tight is called \emph{loose}.
\end{definition}
\begin{remark}
	\label{rem:broad-tight}
	Our results in this section continue to hold for a slightly more general
	definition, where we replace the linear functions $\gamma_jx + \nu_j$ by
	arbitrary monotonically increasing functions $f_j(x)$.
	We focus on linear functions for simplicity and clarity.
	(Also see \Cref{rem:extending-channels}.)
\end{remark}
If the rest of the row is kept fixed, a max-tight entry cannot be increased
without violating privacy constraints, but it can be decreased.
Similarly, a min-tight entry cannot be decreased without violating privacy
constraints, but it can be increased.
Loose entries can be either increased or decreased without violating privacy
constraints.
These perturbations need to be balanced by adjusting other entries in the same
column to satisfy column stochasticity; for example, a max-tight entry can be
decreased while simultaneously increasing a min-tight or loose entry in the
same column.
This is formalized below:

\begin{condition}[Mass can be transferred from entries that are not tight]
	\label{def:mass_transfer}
	Let $\cC$ be a set of channels from $[k]$ to $[\ell]$.
	Let $\bT$ be any extreme point of $\cC$.
	Suppose there are two rows $(r, r')$ and two columns $(c, c')$ (in the display
	below, we take $r < r'$ and $c<c'$ without loss of generality) with values
	$(m,m', M, M')$, as shown below:
	\begin{align*}
		\left[
			\begin{matrix}
				\cdots & \cdots &
				\cdots & \cdots & \cdots\\ \cdots & M & \cdots & m & \cdots \\ \cdots & \cdots
				& \cdots & \cdots & \cdots \\ \cdots & m' & \cdots & M' & \cdots \\ \cdots &
				\cdots & \cdots & \cdots & \cdots
			\end{matrix}
			\right],
	\end{align*}
	such that:
	\begin{itemize}
		\item $\bT(r,c)$ and $\bT(r', c')$ are not min-tight ($M$ and $M'$ above,
		      respectively).
		\item $\bT(r,c')$ and $\bT(r',c)$ are not max-tight ($m$ and $m'$ above,
		      respectively).
	\end{itemize}

	Then there exist $\epsilon' > 0$ and $\delta' > 0$ such that for all $\epsilon
		\in [0,
		\epsilon')$ and $\delta \in [0, \delta')$, the following matrix $\bT'$
	also belongs to $\cC$:
	\begin{align*}
		\bT' = \left[
			\begin{matrix}
				\cdots &
				\cdots & \cdots & \cdots & \cdots\\ \cdots & M - \epsilon & \cdots & m + \delta
				& \cdots \\ \cdots & \cdots & \cdots & \cdots & \cdots \\ \cdots & m' +
				\epsilon & \cdots & M' - \delta & \cdots \\ \cdots & \cdots & \cdots & \cdots &
				\cdots
			\end{matrix}
			\right]\,\,,
	\end{align*}
	where the omitted entries of $\bT$
	and $\bT'$ are the same.
\end{condition}

We show that the channels from \Cref{def:family-of-channels} satisfy
\Cref{def:mass_transfer} in \Cref{sec:properties_of_private_channels}.
Using \Cref{def:mass_transfer},
we show that the following structure is forbidden in
the channels that lead to extreme points of the joint range:
\begin{lemma}
	\label{lem:forbidden-structure}
	Let $p$ and $q$ be two distributions on $[k]$.
	Let $\cC$ be the set of LP channels from \Cref{def:family-of-channels} (or,
	more generally, a convex set of channels satisfying \Cref{def:mass_transfer})
	from $[k]$ to $[\ell]$.
	Suppose $p_i/q_i$ is strictly increasing in $i$.
	Let $\bT \in \cC$ have the following structure: there are two rows $(r, r')$ (in
	the display below, $r < r'$ is taken without loss of generality) and three
	columns $i_1 < i_2 < i_3$ with values $(m,m', m'', M, M', M'')$, as shown below:
	\begin{align*}
		\left[
			\begin{matrix}
				\cdots & \cdots & \cdots & \cdots & \cdots
				& \cdots & \cdots \\ \cdots & M & \cdots & m & \cdots & M' & \cdots \\ \cdots &
				\cdots & \cdots & \cdots & \cdots & \cdots & \cdots \\ \cdots & m' & \cdots &
				M'' & \cdots & m'' & \cdots \\ \cdots & \cdots & \cdots & \cdots & \cdots &
				\cdots & \cdots \\
			\end{matrix}
			\right],
	\end{align*}
	such that:
	\begin{itemize}
		\item $\bT(r,i_1), \bT(r, i_3),$ and $\bT(r', i_2)$ are not min-tight ($M, M',$
		      and $M''$ above, respectively).
		\item $\bT(r,i_2), \bT(r',i_1),$ and $\bT(r', i_3)$ are not max-tight ($m, m',$
		      and $m''$ above, respectively).
	\end{itemize}
    Let $\cA := \{(\bT p, \bT q): \bT \in \cC\}$.
	Then $(\bT p, \bT q)$ cannot be an extreme point of $\cA$.
\end{lemma}
\begin{proof}
    Firstly, the set $\cA$ is convex since $\cC$ is convex.
	For some $\epsilon > 0$ and $\delta > 0$ to be decided later, consider the
	following perturbed matrices:
	\begin{align*}
		\bT' &= \left[
			\begin{matrix}
				\cdots & \cdots & \cdots & \cdots & \cdots & \cdots & \cdots \\
				\cdots & M - \epsilon & \cdots & m + \delta& \cdots & M' & \cdots \\ \cdots &
				\cdots & \cdots & \cdots & \cdots & \cdots & \cdots \\ \cdots & m' + \epsilon &
				\cdots & M'' - \delta & \cdots & m'' & \cdots \\ \cdots & \cdots & \cdots &
				\cdots & \cdots & \cdots & \cdots \\
			\end{matrix}
			\right], \\ 
		\bT'' &= \left[
			\begin{matrix}
				\cdots & \cdots & \cdots & \cdots & \cdots & \cdots & \cdots \\
				\cdots & M & \cdots & m + \epsilon'& \cdots & M' -\delta' & \cdots \\ \cdots &
				\cdots & \cdots & \cdots & \cdots & \cdots & \cdots \\ \cdots & m' & \cdots &
				M'' - \epsilon' & \cdots & m'' + \delta'& \cdots \\ \cdots & \cdots & \cdots &
				\cdots & \cdots & \cdots & \cdots \\
			\end{matrix}
			\right].
	\end{align*}
	To be
	specific, the entries of $\bT'$, $\bT''$, and $\bT$ match except in the six
	locations highlighted here.
    Since $\cC$ satisfies \Cref{def:mass_transfer} (see \Cref{cl:LP-channels-condition}), both $\bT'$ and $\bT''$ belong to the set $\cC$ if
	$\epsilon$, $\epsilon'$, $\delta$, and $\delta$ are small enough and positive.
	We will now show that there exist choices of these parameters such that
	$(\bT p, \bT q)$ is a convex combination of $(\bT' p, \bT' q)$ and $(\bT'' p,
		\bT'' q)$, and these three points are distinct elements of $\cA$.
Consequently, $(\bT p, \bT q)$ will not be an extreme point of $\cA$.
    
	For any $j \in \ell$, let $v_j$ denote the vector in $\R^\ell $ that is $1$ at
	the
	$j^{\text{th}}$ location and $0$ otherwise.
	Define $\theta_i := p_i/q_i$ to be the likelihood ratio.
	If $\theta_i < \infty$, then $p_i = \theta_i q_i$.
    Since $\theta_i$ is strictly increasing in $i$, only $\theta_{i_3}$ may be infinity.
	Let us first suppose that $\theta_{i_3} < \infty$.
	We will consider the case when $\theta_{i_3}$ might be infinity in the end.

	Let us begin by analyzing how $\bT'$ transforms $p$ and $q$.
	Since $\bT'$ differs from $\bT$ only in the four locations mentioned above, $\bT '
		p$ and $\bT p$, both of which are distributions on $[\ell]$, differ only in the
	elements $r$ and $r'$ of $[\ell]$.
	On the element $r$, $(\bT' q)_{r} - (\bT q)_{r}$ is equal to $ - \epsilon
		q_{i_1} +
		\delta q_{i_2} $, and equal to its negation on the element $r'$.
	In particular, they satisfy the relation
	\begin{align*}
		\bT' q &= \bT q + (-\epsilon
		q_{i_1} + \delta q_{i_2} )  \left( v_{r} - v_{r'} \right).
	\end{align*}
	If $ \epsilon q_{i_1} = \delta q _{i_2}$, we have $\bT'q = \bT
		q$.
	Under the same setting, $p$ is transformed as follows:
	\begin{align*}
		\bT' p &=
		\bT p + (-\epsilon p_{i_1} + \delta p_{i_2} ) \left( v_{r} - v_{r'} \right)
		\\ &= \bT p + (-\epsilon \theta_{i_1} q_{i_1} + \delta \theta_{i_2} q_{i_2} )
		 \left( v_{r} - v_{r'} \right)\\ &= \bT p + \epsilon
		q_{i_1}(-\theta_{i_1} + \theta_{i_2})  \left( v_{r} - v_{r'} \right).
	\end{align*}
	We now analyze the effect of $\bT''$,
	which satisfies
	\begin{align*}
		\bT'' q &= \bT q + (\epsilon'
		q_{i_2} - \delta' q_{i_3} ) \left( v_{r} - v_{r'} \right).
	\end{align*}
	If $ \epsilon' q_{i_2} = \delta' q _{i_3}$, we have $\bT''q = \bT q$.
	Under the same setting, $p$ is transformed as follows:
	\begin{align*}
		\bT'' p
		&= \bT p + (\epsilon' p_{i_2} - \delta' p_{i_3} ) \left( v_{r} - v_{r'}
		\right) \\ &= \bT p + (-\epsilon' \theta_{i_2} q_{i_2} + \delta' \theta_{i_3}
		q_{i_3} )  \left( v_{r} - v_{r'} \right)\\ &= \bT p + \epsilon'
		q_{i_2}(-\theta_{i_2} + \theta_{i_3})  \left( v_{r} - v_{r'} \right).
	\end{align*}
	Now observe that $\theta_{i_1} < \theta_{i_2} < \theta_{i_3}$.
	By choosing $\epsilon > 0$ and $\epsilon'>0$ small enough such that
	$\epsilon q_{i_1} (- \theta_{i_1} + \theta_{i_2}) = \epsilon' q_{i_2} (-
		\theta_{i_2} + \theta_{i_3})$, we obtain
	\begin{align*}
		(\bT p , \bT q) =
		\frac{1}{2}\cdot \left( \bT' p , \bT' q\right) + \frac{1}{2}\cdot \left( \bT' p
		, \bT' q\right),
	\end{align*}
	and all three points are distinct elements of
	$\cA$.
    Such a choice of $\epsilon$ and $\epsilon'$ always exists, since both $q_{i_1} (- \theta_{i_1} + \theta_{i_2})$ and $q_{i_2} (-
		\theta_{i_2} + \theta_{i_3})$ are positive and finite. 
	Thus, $(\bT p, \bT q)$ is not an extreme point of $\cA$.

	Let us now consider the case when $\theta_{i_3} = \infty$, or equivalently,
	$q_{i_3} = 0$.
	Define $\epsilon'$ to be $0$, so that $\bT''q = \bT q$ and $\bT'' p = \bT p
		- \delta' p_{i_3} \left( v_{r} - v_{r'} \right)$.
	Then choose $\delta'> 0$ and $\epsilon > 0$ small enough such that $\epsilon
		q_{i_1} \left( \theta_{i_2} - \theta_{i_1} \right) = \delta' p_{i_3}$, which is
	possible since both sides are positive and finite.
	Thus, $(\bT p, \bT q)$ is a non-trivial convex combination of
	$(\bT 'p, \bT' q)$ and $(\bT ''p, \bT'' q)$, so is not an extreme point of
	$\cA$.
\end{proof}

\subsubsection{Proof of \Cref{thm:ext-point-priv-uniq-col-poly}}
\label{sec:proof-lem-ext-point-uniq-poly}

\begin{proof}
    Without loss of generality, we assume that the likelihood ratios $p_i/q_i$ are unique and strictly increasing in $i$. 
    We refer the reader to the proof of \Cref{thm:ext-point-comm} and \Cref{claim:degenerate-p-q} for more details.

	Let $\bT \in \cC$ be a channel from $[k]$ to $[\ell]$ such that $(\bT p
		,\bT q)$ is an extreme point of $\cA$.
	Suppose that there are $\ell '$ unique columns in $\bT$ with $\ell ' > 2\ell
		^2$.
	From now on, we assume that $\ell ' = 2\ell^2$; otherwise, we apply the following
	argument to
	the first $2 \ell^2$ distinct columns.

	Let $c, c' \in [k]$ be such that the $c^{\text{th}}$ and $c'^{\text{th}}$ columns of $\bT$ are distinct.
	Observe that for every pair of distinct columns $c$ and $c'$, 
     there are two
	rows such that $c^{\text{th}}$ column has a strictly bigger value than the $c'^{\text{th}}$
	column on one row, and vice versa on the another row.
	This is because both of the columns sum up to $1$, so if a particular
	column has a larger entry in a row, its entry must be smaller in a different
	row.
	In particular, there exist two rows $g(c,c')$ and $h(c, c')$ such that
	$ \bT(g(c,c'), c) > \bT(g(c,c'), c')$ and $ \bT(h(c, c'), c ) < \bT(h(c,
			c'), c')$.
	As a result,
	$\bT(g(c,c'), c)$ and $\bT(h(c,c'), c')$ are not min-tight, and $\bT(g(c,c'),
			c')$ and $\bT(h(c,c'), c)$ are not max-tight.

	We now order the distinct columns of $\bT$ in the order of their appearance
	from left to right.
	Let $i_1, i_2, \ldots, i_{\ell '}$ be the indices of the unique columns.
	For example, the first distinct column $i_1$ is the first column of $\bT$
	(corresponding to the element $1$).
	The second distinct column $i_2$ is the first column of $\bT$ that is different
	from the first column.
	The third distinct column is the first column of $\bT$ that is different from
	the first two columns.
	Let $\cI$ be the set of unique column indices of $\bT$.

	Now, we divide the distinct columns in $\bT$ into pairs: $\cH =
		\{(i_1,i_2), (i_3,i_4), \ldots, (i_{\ell '-1}, i_{\ell '})\}$.
	The total number of possible choices in $\cH$ is $\ell'/2$, and for
	every $(m,m+1)$ in $\cH$, the possible number of choices of $\left( g(i_m,
		i_{m+1}), h(i_m, i_{m+1})\right)$ is at
	most $\ell (\ell -1)$, since both of these lie in $[\ell]$ and are distinct.
	Thus, there must exist two pairs in $\cH$ whose corresponding indices are the
	same, since $\frac{\ell'}{2} = \ell^2 > \ell(\ell-1)$.

	Without loss of generality, we let these pairs of columns be $(i_1,i_2)$
	and $(i_3,i_4)$.
	Let $r := g(i_1, i_{2}) = g(i_3, i_{4})$ and $r' := h(i_1, i_{2}) = h(i_3,
		i_{4})$.
	Then the previous discussion implies that:
	\begin{itemize}
		\item
		      $\bT(r,i_1)$ and $\bT(r,i_3)$ are not min-tight, and $\bT(r',i_1)$ and
		      $\bT(r',i_3)$ are not max-tight.

		\item $\bT(r,i_2)$ and $\bT(r,i_4)$ are not max-tight, and $\bT(r',i_2)$ and
		      $\bT(r',i_4)$ are not min-tight.
	\end{itemize}
	Thus, the columns $i_1, i_2$, and $i_3$ satisfy the conditions of
	\Cref{lem:forbidden-structure}, i.e., they exhibit the forbidden structure.
	This implies that $(\bT p, \bT q)$ cannot be an extreme point of $\cA$.
	Therefore, $\ell ' \leq 2\ell ^2$.
\end{proof}

\subsection{Proof of \Cref{thm:ext-point-priv-comm}: Unique Columns to Threshold Channels}
\label{sec:unique_columns_to_threshold_channels}
In this section, we provide the proof of \Cref{thm:ext-point-priv-comm} using \Cref{thm:ext-point-priv-uniq-col-poly}.
Noting that our main structural result is more widely applicable
(\Cref{def:mass_transfer}),
we prove a more general version of \Cref{thm:ext-point-priv-comm} below for
\Cref{def:family-of-channels}.
Before doing so, we require an additional property on the
set of our channels, proved in \Cref{sec:properties_of_private_channels}:
\begin{restatable}[Closure under pre-processing]{claim}{ClComposition}
	\label{cl:composition}
	The set $\cJ_{\ell,k}^{\gamma, \nu}$ satisfies the following closure property under
	pre-processing:
	\begin{align}
	\label{eq:composition-pre-processing}
	\cJ_{\ell,k}^{\gamma, \nu} = \bigcup_{\ell'=1}^k \left\{\bT_2 \times \bT_1 : \bT_2 \in \cJ_{\ell, \ell'}^{\gamma, \nu} \textnormal{
			and } \bT_1 \in \cT_{\ell', k} \right\}.
	\end{align}
\end{restatable}
Informally, if we take an arbitrary channel $\bT_1$ and compose it with an LP private channel $\bT_2$, the composition $\bT_2 \times \bT_1$ is also LP private.

The following result is thus a more general version of
\Cref{thm:ext-point-priv-comm}:
\begin{theorem}[Structure of optimal channels]
	\label{thm:OptStructGeneral}
	Let $p$ and $q$ be distributions on $[k]$.
	For any $\ell \in \N$, let $\cC$ be the set of channels $\cJ_{\ell,k}^{\gamma,
			\nu}$ from \Cref{def:family-of-channels}.
	Let $\cA$ be the set of all pairs of distributions that are obtained by
	applying a channel from $\cC$ to $p$ and $q$, i.e.,
	\begin{align}
		\cA = \{(\bT
		p, \bT q) \mid \bT \in \cC\}.
	\end{align}
	If $(\bT p, \bT q)$ is an extreme point of $\cA$, then $\bT$ can be
	written as $\bT = \bT_2 \times \bT_1$, for some $\bT_1 \in \cTT_{\ell',k}$ and
	$\bT_2$ an extreme point of the set $\cJ_{\ell, \ell'}^{\gamma, \nu}$.
\end{theorem}
\begin{proof}
	Since $\cC$ is convex, the joint range $\cA$ is convex.
	By \Cref{thm:ext-point-priv-uniq-col-poly}, we know that if
	$(\bT p, \bT q)$ is an extreme point of $\cA$, then $\bT$ can
	be written as $\bT_2 \times \bT_1$, where $\bT_1 \in \cT_{\ell ',k}$ for $\ell
		' := 2\ell ^2$.
	Using \Cref{cl:composition},
	any such channel in $\cC$ is of the form $\bT_2 \times \bT_1$,
	where $\bT_1 \in \cT_{\ell ',k}$ and $\bT_2 \in \cJ_{\ell, \ell'}^{\gamma,
			\nu}$.
	Combining the last two observations, we obtain the following:
	\begin{align}
	\label{eq:conv-ext}
	\cA = \conv\left( \left\{(\bT_2 \times \bT_1 p, \bT_2 \times \bT_1 q): \bT_2 \in  \cJ_{\ell, \ell'}^{\gamma, \nu}, \bT_1 \in \cT_{\ell',k} \right\} \right).
	\end{align}
	We now claim that we can further take $\bT_1$ to be a threshold channel $\bT_1
		\in \cTT_{\ell',k}$ and $\bT_2$ to be an extreme point of $\cJ_{\ell,
			\ell'}^{\gamma, \nu}$.
	This claim follows if we can write an arbitrary point in $\cA$ as a convex
	combination of elements of
	the set $\left\{(\bT_2 \times \bT_1 p, \bT_2 \times \bT_1 q): \bT_2 \in \ext\left(
		\cJ_{\ell, \ell'}^{\gamma, \nu} \right)
		, \bT_1 \in \cTT_{\ell',k} \right\}$.
	By equation~\Cref{eq:conv-ext}, it suffices to demonstrate this convex combination for
	all points of the form $(\bT_2 \times \bT_1 p, \bT_2 \times \bT_1 q)$, for
	some $\bT_2 \in \cJ_{\ell, \ell'}^{\gamma, \nu}$ and $\bT_1 \in \cT_{\ell',k}$.

	Let $\bH_1, \bH_2, \dots$ be extreme points of $\cJ_{\ell, \ell'}^{\gamma,
			\nu}$, and let $\bL_1, \bL_2, \dots$ be an enumeration of the threshold channels
	$\cTT_{\ell',k}$.
	By definition, any $\bT_1 \in \cJ_{\ell, \ell'}^{\gamma,
			\nu}$ can be written as $\sum_i \alpha_i \bH_i$ for some convex combination
	$\alpha_1, \alpha_2, \dots$.
	Furthermore, \Cref{thm:ext-point-comm} implies that any $(\bT_2 p ,\bT_2 q)$,
	for $\bT_2 \in \cT_{\ell',k}$, can be written as $\sum_{j} \beta_j (\bL_j p,
		\bL_j q) = (\sum_{j} \beta_j \bL_j p, \sum_{j} \beta_j \bL_j q)$, for some
	convex combination $\beta_1, \beta_2, \dots$.

	Thus, any arbitrary point $(\bT_2 \times \bT_1 p, \bT_2 \times \bT_1 q)$, for
	$\bT_2 \in \cJ_{\ell, \ell'}^{\gamma, \nu}$ and $\bT_1 \in \cT_{\ell',k}$, can
	be written as
	\begin{align*}
	\left(\bT_2 \times \bT_1 p, \bT_2 \times \bT_1 q\right) &= \left( \left( \sum_i \alpha_i \bH_i \right) \times \bT_1 p,  \left( \sum_i \alpha_i \bH_i \right)\times \bT_1 q\right) \\
	&= \sum_i \alpha_i \left(   \bH_i \times \bT_1 p,  \bH_i \times \bT_1 q\right) \\
	&= \sum_i \alpha_i \left(   \bH_i \left( \bT_1 p \right),  \bH_i  \left( \bT_1 q \right)\right) \\
	&= \sum_i \alpha_i \left(   \bH_i \left( \sum_{j} \beta_j \bL_j p \right),  \bH_i \left( \sum_{j} \beta_j \bL_j q \right)\right) \\
	&= \sum_i \alpha_i \left(   \sum_{j} \beta_j \bH_i \times  \bL_j p ,  \sum_{j} \beta_j \bH_i \times \bL_j q \right) \\
	&= \sum_i\sum_j \alpha_i  \beta_j\left(    \bH_i \times  \bL_j p ,  \bH_i \times \bL_j q \right).
	\end{align*}
	Finally, note that $\{\alpha_ i \beta_j\}$ are also valid convex combinations of
	$(\bH_i \times \bL_j p , \bH_i \times \bL_j q)$, since they are nonnegative and
	sum to $1$.
\end{proof}

\begin{remark}[Extending \Cref{thm:OptStructGeneral,thm:ext-point-priv-comm} to a more general set of constraints]
	\label{rem:extending-channels}
	We note that \Cref{thm:OptStructGeneral} can be extended to an arbitrary convex
	set of channels $\cC$ that satisfy (appropriately modified versions of)
	\Cref{def:mass_transfer} and equation~\Cref{eq:composition-pre-processing}.
	(Also see \Cref{rem:broad-tight}.)
\end{remark}
\subsection{Application to Hypothesis Testing}
\label{SecLDPCommTesting}

In \Cref{sec:hypothesis_testing}, we showed that the minimax-optimal sample complexity
can be obtained by a communication-efficient and efficiently computable channel, up to logarithmic factors.
However, for a particular $(p,q)$, these guarantees can be significantly
improved.
For example, consider the extreme case when $p$ and $q$ are the following two
distributions on $[k]$: for $\gamma$ small enough, 
\begin{align*}
p & = [\alpha, 1 - \alpha - (k-2)\gamma, \gamma, \gamma, \dots, \gamma], \\
q & = [\beta, 1 - \beta- (k-2)\gamma, \gamma, \gamma, \dots, \gamma].
\end{align*}
Let $\bT'$ be a deterministic binary channel that maps the first and second elements to different elements, while assigning the remaining elements
arbitrarily.
Now consider the following private channel $\bT$: the channel $\bT'$, followed
by the randomized response over binary distributions.
Then as $\gamma \to 0$, the performance of $\bT$ mirrors equation~\Cref{eq:sample-comp-binary}, which is much better than the minimax bound of
equation~\Cref{eq:worst-case-samp-comp}.
Thus, there is a wide gap between instance-optimal and minimax-optimal
performance. 
We thus consider the computational question of optimizing a quasi-convex function
$g(\bT p, \bT q)$
over all possible $\epsilon$-private channels that map to a domain of size
$\ell $.

The following result proves \Cref{cor:alg-priv-quasi-cvx} for $\cC$ equal to
$\cP^{\epsilon}_{\ell, k}$:
\begin{restatable}[Computationally efficient algorithms for maximizing quasi-convex functions under privacy constraints]{corollary}{CorOptimizationPolySecond}
	\label{cor:alg-priv-quasi-cvx-second}
	Let $p$ and $q$ be fixed distributions over $[k]$, let
	$\cC $ be the set of channels $\cJ_{\ell,k}^{\gamma, \nu}$ from
	\Cref{def:family-of-channels}, and let $\cA = \{(\bT p, \bT q): \bT \in \cC\}$.
	Let $g: \cA \to \R$ be a jointly quasi-convex function.
	Then there is an algorithm that solves
	$\max_{\bT \in \cC} g(\bT p ,\bT q)$
	in	time polynomial in $k^{\ell^2}$ and $2^{O(\ell^3 \log \ell)}$.\footnote{Recall that $g$ is assumed to be permutation invariant. If not, an extra factor of $\ell!$ will appear in the time complexity.}
\end{restatable}
\begin{proof}
	The algorithm is as follows: we try all threshold channels in $\bT_1 \in
		\cTT_{\ell',k}$ and all extreme points of $\cJ^{\gamma,\nu}_{\ell,\ell'}$, and output the
	channel $\bT = \bT_2 \times \bT_1$ that attains the maximum value of $g (\bT p
		, \bT q)$.
	By \Cref{thm:OptStructGeneral} and quasi-convexity of $g$, we know the algorithm will output a correct
	value, since all extreme points are of this form.
	Thus, we focus on bounding the runtime.
	We know that the cardinality of $ \cTT_{\ell ',k}$ is bounded by $k^{\ell '}$
	(up to a rotation of output rows).
	By \Cref{fact:vertex-enumeration}, the time taken to iterate through all the
	extreme points of $\epsilon$-LDP channels from $[\ell']$ to $[\ell]$ is at most
	polynomial in $2^{\ell^3 \log \ell}$, since $\cJ_{\ell,\ell'}$ is a polytope
	in $\R^{2\ell^3}$ with $\poly(\ell)$ inequalities.
	This completes the proof.
\end{proof}

The proof of \Cref{cor:alg-priv-sbht} is immediate from \Cref{fact:sample-complexity},
\Cref{cor:alg-priv-quasi-cvx}, and \Cref{lem:comm-privacy-effect}, stated later.

\section{Extensions to Other Notions of Privacy} %
\label{sec:extensions_to_other_privacy}

In this section, we explore computational and statistical aspects of hypothesis testing under other notions of privacy.
	\Cref{sec:approximate_local_privacy} is on approximate privacy, in which we first focus on $(\epsilon, \delta	)$-LDP and then our proposed definition of approximate privacy.
	Next, we focus on binary communication constraints for R\'{e}nyi
	differential privacy in \Cref{sec:other-priv-renyi}.
	This will be possible since our algorithmic and structural results were not
	restricted to the case of pure LDP.

We begin by noting that communication constraints have a benign
	effect on the sample complexity of hypothesis testing for many notions of privacy:

\begin{condition}[Closure under post-processing]
	\label{cond:l-postprocessing}
	Let $ k \in \N$. For each $r \in \N$, consider sets $\cC_r \subseteq \cT_{r, k}$ and define $\cC = \cup_{r \in \N} \cC_r$.
	We say $\cC$ \emph{satisfies $\ell$-post-processing} if for every $r \in \N$, if $\bT
		\in \cC_r$ and $\bH$ is a deterministic channel from $[r]$ to $[\ell] $,
	the channel $\bH \times \bT$ also belongs to $\cC_\ell$, and thus to $\cC$.
\end{condition}
Post-processing is satisfied by various notions of privacy: $\epsilon$-pure
privacy,
$(\epsilon,\delta)$-approximate privacy (see Dwork and Roth~\cite[Proposition 2.1]{DworkRoth13}), and R\'{e}nyi privacy~\cite{Mir17}.
For a set of channels $\cC$,
we use the notation $\nstar(p,q,\cC)$ to denote the sample complexity of
hypothesis testing under channel constraints of $\cC$ in \Cref{def:shtgeneric}.
The following result shows that even with binary communication constraints, the
sample complexity increases by at most a logarithmic factor:

\begin{proposition}[Benign effect of communication constraints on sample complexity under closure]
	\label{lem:comm-privacy-effect}
	Let $p$ and $q$ be any two distributions on $[k]$.
	Let $\cC$ be a set of channels that satisfy $\ell$-post-processing
	(\Cref{cond:l-postprocessing}) for some $\ell > 1$.
	Let $\cC_\ell$ denote the subset of channels in $\cC$ that map to a domain of
	size $\ell$.
	Then
	\begin{align}
    \nstar(p,q,\cC_\ell) \lesssim \nstar(p,q,\cC) \cdot \left(1 + \frac{\log \left(\nstar(p,q,\cC)\right)}{\ell}\right).  
\end{align}
\end{proposition}
\begin{proof}
	Let $\bT$ be the optimal channel in $\cC$ that maximizes $\hel^2(\bT p, \bT
		q)$.
	Let $k'$ be the size of the range of $\bT$.
	By \Cref{fact:id-channel}, we have $\nstar(p,q,\cC)\asymp 1/{\hel^2(\bT p,
		\bT q)}$.
	By \Cref{fact:comm-constraints}, we know that there exists $\bT' \in
		\cT_{\ell,k'}$ 
  such that\footnote{If the supremum is not attained, the proof can be modified by considering a suitable sequence of channels and applying a similar argument.}
	\begin{align}
    \hel^2(\bT p, \bT q)    \lesssim  \hel^2(\bT' ( \bT p), \bT' (\bT q))\cdot 
     \left(1 + \frac{\log(1/\hel^2(\bT p,
     \bT q))}{\ell}\right) .
    \end{align}
	By the assumed closure of $\cC$ under post-processing,
	the channel $\bT'\times\bT $ belongs to $\cC$. Thus, the channel $\bT'\times\bT $
	also belongs to $\cC_\ell$, since its output is of size $\ell$.
	This implies that the sample complexity $\nstar(p,q,\cC_\ell)$
	is at most $1/\hel^2(\bT'\times\bT p, \bT'\times\bT q)$.
	Using the fact that $\nstar(p,q,\cC)\asymp 1/{\hel^2(\bT p, \bT q)}$, we obtain the
	desired result.
\end{proof}

Thus, in the rest of this section, our main focus will be on the setting of
binary channels.

\subsection{Approximate Local Privacy} 
\label{sec:approximate_local_privacy}
In this section, we first focus on $(\epsilon,\delta)$-approximate LDP (\Cref{def:ldp-approx}).
We begin by showing upper bounds on the associated sample complexity.
On the computational front, we present efficient algorithms for the case of binary
constraints and then propose a relaxation for the case of larger output
domains.

We first recall the definition of $(\epsilon,\delta)$-LDP:
\begin{definition}[($\epsilon, \delta$)-LDP]
	\label{def:ldp-approx}
	We say a channel from $\cX$ to $\cY$ is ($\epsilon, \delta$)-\emph{LDP} if for all $S
		\subseteq
		\cY$, we have
	\begin{align}\label{epsilon delta LDP inequality}
		\sup_{x,x' \in \cX}  \P [\bT(x) \in S)]  -
		e^{\epsilon} \cdot \P [\bT(x) \in S)] - \delta \leq 0.
	\end{align}
\end{definition}
What makes the analysis of $(\epsilon,\delta)$-LDP different from
$\epsilon$-LDP is that when $|\mathcal{Y}|>2$, the condition in inequality~\eqref{epsilon
	delta LDP inequality} should be verified for \emph{all} sets $S\subseteq \mathcal{Y}$, not just singleton sets ($|S|=1$).
Only when $|\mathcal{Y}|=2$ is it enough to consider singleton sets $S$.

Let $\nstar(p,q, \left( \epsilon,\delta \right))$ denote the sample complexity
for the setting in \Cref{def:sht-ldp}, with $\cC$ equal to the set of all
$(\epsilon, \delta)$-LDP channels.
We directly obtain the following upper bound on the sample complexity, proved in \Cref{app:other_notions_of_privacy}, which
happens to be tight for the case of binary distributions:
\begin{restatable}[Sample complexity of approximate LDP]{claim}{LemSampCompApproxLDP}
	For all $\delta \in (0,1)$, we have
	\begin{align*}
\nstar(p,q, \left( \epsilon,\delta \right)) \lesssim \min\left(\nstar(p,q, \epsilon) \cdot \frac{1}{1 - \delta}\,\,, \nstar(p,q) \cdot \frac{1}{\delta}  \right).
\end{align*}
	Moreover, this is tight (up to constant factors) when both $p$ and $q$ are
	binary distributions.
\end{restatable}
In the rest of this section, we focus on efficient algorithms in the presence
of both privacy and communication constraints.

Turning to computationally efficient algorithms for the case of privacy and
communication constraints,
we present two kinds of results:
exact results for the case of binary outputs, and sharp relaxations for the case
of multiple outputs.
\paragraph{Binary channels:}
Let $\cC$ be the set of all $(\epsilon, \delta)$-approximate LDP channels
from $[k]$ to $[2]$, i.e., binary channels.
Let $\gamma = (e^\epsilon, e^\epsilon	)$ and $\nu = (\delta, \delta )$.
Observe that $\cC$ is then equal to $\cJ_{2,k}^{\gamma,\nu}$, defined in
\Cref{def:family-of-channels}.
Thus, \Cref{thm:OptStructGeneral,cor:alg-priv-quasi-cvx-second} hold in this
case, as well.

\paragraph{Channels with larger output spaces:}
Here,
we define a new notion of privacy that relaxes $(\epsilon, \delta)$-LDP. It is enough to verify whether the privacy condition holds for singleton events $S$:
\begin{definition}[($\epsilon,\delta$)-SLDP]
	\label{def:sldp}
	We say a channel $\cX$ to $\cY$ is ($\epsilon, \delta$)-\emph{singleton-based-LDP ($(\epsilon, \delta)$-SLDP)}
 if for all $S \subseteq
		\cY$, we have
	\begin{align*}
		\sup_{x,x' \in \cX}  \P [\bT(x) \in S)]  -
		e^{\epsilon} \cdot \P [\bT(x) \in S)] - \delta\cdot |S| \leq 0.
	\end{align*}
\end{definition}
The following result shows that $(\epsilon, \delta)$-SLDP is a good
approximation to
$(\epsilon,\delta)$-LDP when the output space is small:
\begin{claim}[Relations between LDP and SLDP] Consider a channel $\bT$ from $\cX$ to $[\ell]$.
	\begin{enumerate}
		\item
		      If $\bT$ is $(\epsilon,\delta)$-SLDP, it is
		      $(\epsilon, \ell  \delta)$-LDP.
		\item
		      If $\bT$ is $(\epsilon,\delta)$-LDP, it is
		      $(\epsilon,
			      \delta )$-SLDP. %

	\end{enumerate}
\end{claim}
The proof is immediate from the definitions of $(\epsilon, \delta)$-LDP and $(\epsilon, \delta)$-SLDP, and we omit it. We now show that it is easy to optimize over SLDP channels in the presence of
communication constraints.
For any $\ell	\in \N$,
let $\cC$ be the set of all channels from $[k]$ to $[\ell]$ that satisfy
$(\epsilon,\delta)$-SLDP.
Let $\gamma = (e^\epsilon, e^\epsilon, \dots, e^\epsilon)$
and $\nu = (\delta, \delta, \dots, \delta )$.
Observe that $\cC$ is then equal to $\cJ_{\ell, k}^{\gamma,\nu}$, defined in
\Cref{def:family-of-channels}.
Thus, \Cref{thm:OptStructGeneral,cor:alg-priv-quasi-cvx-second} imply that we
can efficiently optimize over SLDP channels.

\subsection{Other Notions of Privacy}
\label{sec:other-priv-renyi}
We briefly note that our computationally efficient algorithms hold for a wider
family of channels defined in \Cref{def:family-of-channels}; see also
\Cref{rem:extending-channels}.

Finally, we consider the case of R\'{e}nyi differential privacy introduced in Mironov~\cite{Mir17}: 
\begin{definition}[$(\epsilon, \alpha)$-R\'{e}nyi differential privacy] Let $\epsilon \in \real_+$ and $\alpha > 1$, and let $\cX$ and $\cY$ be two domains. A channel $\bT: \cX \to \cY$ satisfies $(\epsilon, \alpha)$-\emph{RDP} if for all $x, x' \in \cX$, we have
$$D_\alpha(\bT(x)\|\bT(x')) \le \epsilon,$$
where $D_\alpha(p \| q)$ is the R\'{e}nyi divergence of order $\alpha$ between two distributions $p$ and $q$ on the same probability space, defined as
$$D_\alpha(p \| q) := \frac{1}{\alpha-1} \log \E_{X \sim q}\left[ \left(\frac{p(X)}{q(X)}\right)^\alpha\right].$$ 
\end{definition}
R\'{e}nyi divergence is also defined for $\alpha = 1$ and $\alpha = \infty$ by taking limits. When $\alpha = 1$, the limit yields the Kullback--Leibler divergence, and when $\alpha = \infty$, it leads to the supremum of the log-likelihood ratio between $p$ and $q$. In fact, $(\infty, \epsilon)$-RDP is identical to $\epsilon$-LDP. Similarly, $(1, \epsilon)$-RDP is closely related to mutual information-based privacy~\cite{CuffYu16}, since the corresponding channel $\bT$ has Shannon capacity at most $\epsilon$.
\begin{proposition}[R\'{e}nyi differential privacy and binary constraints]
	Let $\epsilon > 0$ and $\alpha >1$. Let $\cC$ be the set of $(\epsilon, \alpha)$-RDP channels from $[k]$ to $[2]$.
	Let $p$ and $q$ be two distributions on $[k]$, and define $\cA:= \{(\bT p ,\bT
		q): \bT \in \cC\}$.
	If $(\bT p, \bT q)$ is an extreme point of $\cA$ for $\bT \in \cC$,
	then $\bT$ can be written as $\bT_1 \times \bT_2$, where $\bT_1$ is an extreme point of the set of 
	$(\epsilon, \alpha)$-RDP channels from $[2]$ to $[2]$, and $\bT_2$ is a binary
	threshold channel from $[k]$.
\end{proposition}
\begin{proof}
Consider two binary distributions $[x, 1-x]$ and $[y, 1-y]$, where $0 \le x, y \le 1$. The $\alpha$-R\'{e}nyi divergence between the distributions is given by 
$$D_\alpha(x \| y) := \frac{1}{\alpha-1} \log \left(x^\alpha y^{1-\alpha} + (1-x)^\alpha (1-y)^{1-\alpha}\right).$$ 
Observe that the term inside the logarithm is convex in $y$ for fixed $x$, and is minimized when $y = x$. Hence, the R\'{e}nyi divergence above, as a function of $y$, is decreasing for $y \in [0, x]$ and increasing for $y \in [x, 1]$. A similar conclusion holds for fixed $y$ and varying $x$. 

Consider a channel $\bT \in \cC$ %
given by
\begin{align*}
\bT = \begin{bmatrix}
x_1 &x_2 &\dots &x_k\\
1-x_1 &1-x_2 &\dots &1-x_k
\end{bmatrix}.
\end{align*}
Without loss of generality, assume $x_1 \le x_2 \le \dots \le x_k$ %
Suppose there is an index $j$ such that $x_1 < x_j < x_k$. Observe that $x_j \notin \{0, 1\}$. By the monotonicity property of the R\'{e}nyi divergence noted above, for any index $i$, we have
\begin{align*}
\max\left\{D_\alpha(x_j \| x_i), D_\alpha(x_i \| x_j)\right\} < \max\left\{D_\alpha(x_1 \| x_k), D_\alpha(x_k \| x_1)\right\} \le \epsilon.
\end{align*}
This means that $x_j$ can be perturbed up and down by a small enough $\delta$ such that the R\'{e}nyi divergence constraints continue to be satisfied. Such perturbations will allow $\bT$ to be written as a convex combination of two distinct matrices, so $\bT$ cannot be an extreme point of the (convex) set $\cC$. Thus, an extreme point must have only two distinct columns; i.e., it must have the form 
\begin{align*}
\bT = 
\begin{bmatrix}
x_1 &x_1 &\dots &x_1 &x_k &x_k &\dots &x_k\\
1-x_1 &1-x_1 &\dots &1-x_1 &1-x_k &1-x_k &\dots &1-x_k
\end{bmatrix}.
\end{align*}
Equivalently, any extreme point is a deterministic channel from $[k] \to [2]$ followed by an RDP-channel from $[2]\to[2]$. Since we are only concerned with extreme points that correspond to extreme points of the joint range $\cA$, an argument identical to the one in the proof of \Cref{thm:OptStructGeneral} yields that an extreme point must admit a decomposition $\bT_1 \times \bT_2$, where $\bT_2$ is a threshold channel from $[k] \to [2]$ and $\bT_1$ is an extreme point of the set of RDP channels from $[2] \to [2]$. 
\end{proof}

The above result implies that given a quasi-convex function $g:\cA \to \R$, if
we are interested in maximizing $g(\bT p, \bT q)$ over $\bT \in
	\cC$,
the optimal $\bT$ can be written as $\bT_1 \times \bT_2$, where $\bT_1$ is
a binary-input, binary-output R\'{e}nyi private channel and $\bT_2$ is a threshold
channel.
Since there are only $2k$ threshold channels, we can try all those choices of
$\bT_2$, and then try to optimize over $\bT_1$ for each of those choices.
However, each such problem is over binary inputs and binary outputs, and thus
is amenable to grid search.

\begin{remark}
In addition to the convexity of RDP channels, we also used the closure-under-pre-processing property (see \Cref{cl:composition}) and the unimodality of $D_\alpha(x\|y)$ when one of the variables is fixed and the other is varied. The above proof technique will therefore work for any set of convex channels from $[k] \to [2]$ that are closed under pre-processing, and are defined in terms of such a unimodal function. In particular, our results will continue to hold for all $f$-divergence-based private channels, defined as all $\bT$ satisfying $$D_f(\bT(x) \| \bT(x')) \le \epsilon.$$ 
Our results also hold for zero-concentrated differential privacy (z-CDP)~\cite{BunSte16}, which is a notion of privacy defined using R\'{e}nyi divergences.%

\end{remark}
\section{Conclusion} %
\label{sec:conclusion}

In this paper, we considered the sample complexity of simple binary hypothesis
testing under privacy and communication constraints.
We considered two families of problems: finding minimax-optimal bounds and
algorithms, and finding instance-optimal bounds and algorithms.

For minimax optimality, we considered the set of distributions with fixed
Hellinger divergences and total variation distances.
This is a natural family to consider, because these two metrics characterize the
sample complexity in the low- and high-privacy regimes.
Prior work did not resolve the question of sample complexity in the moderate-privacy regime; our work has addressed this gap in the literature, by establishing a sample-complexity lower bound via a
carefully constructed family of distribution pairs on the ternary alphabet.
Our results highlight a curious separation between the binary and ternary (and
larger alphabet) settings, roughly implying that the binary case is
substantially easier (i.e., has a lower sample complexity) than the general
case.

Our focus on instance optimality sets our paper apart from most prior work on
information-constrained estimation, which exclusively considered minimax
optimality.
When only privacy constraints are imposed, we established approximately
instance-optimal algorithms; i.e., for any distribution pair, we proposed a
protocol whose sample complexity is within logarithmic factors of the true
sample complexity.
Importantly, the algorithm we proposed to identify this protocol is computationally
efficient, taking time polynomial in $k$, the support size of the
distributions.
When both privacy and communication constraints are in force, we developed
instance-optimal algorithms, i.e., protocols whose sample
complexity is within constant factors of the true sample complexity.
As before, these algorithms take time polynomial in $k$, for any constant
communication constraint of size $\ell$.

Our results highlight the critical role played by threshold channels in both
communication- and privacy-constrained settings.
We showed that for any distribution pair, the channel with output size $\ell$
that maximizes the output divergence (Hellinger, Kullback--Leibler, or any
quasi-convex function in general) among all channels with fixed output size
$\ell$ must be a threshold channel.
Furthermore, optimal private channels with output size $\ell$ admit a
decomposition into a threshold channel cascaded with a private channel.
These two results underpin our algorithmic contributions.

There are many interesting open problems stemming from our work that would be
worth exploring.
We did not characterize instance-optimal sample complexity in the moderate-privacy regime; our work shows that it is not characterized in terms of the Hellinger divergence and total
variation distance, but leaves open the possibility of some other
divergence, such as the $E_\gamma$ divergence, capturing the sample complexity.
We identified a forbidden structure for optimal private channels; however,
the best algorithm from Kairouz, Oh, and Viswanath~\cite{KaiOV16} does not use this information at all.
It would be interesting to see if that algorithm could be made
more efficient by incorporating the extra structural information.
Many open questions remain for the approximate LDP setting, as well.
There is no known upper bound on the number of outputs that suffice for
optimal approximate LDP channels.
It is unknown if instance-optimal approximately private channels with
$\ell>2$ outputs admit decompositions into threshold channels cascaded with
private channels, similar to the pure LDP setting. Some early investigations~\cite{ElaJog23} suggest that the set of extreme channels for $(0, \delta)$-private channels with $\ell = 3$ admit a simple description for arbitrary $[k]$. However, even for $\ell = 4$, the structure of extreme points can be quite complex.
It would be interesting to see if optimal SLDP channels, which are efficient to
find, are nearly instance optimal for approximate LDP. Finally, we did not address whether interactive protocols (sequential or blackboard) lead to strictly better sample complexities than the non-interactive ones considered in this paper.

\section*{Acknowledgements}

We thank anonymous reviewers of the conference version of this paper (submitted to the Conference on Learning Theory, 2023) for their helpful feedback.
  \printbibliography
  \appendix

\section{Randomized Response in Low-Privacy Regime}
\label{app:rand-resp-low-priv}
In this section, we prove \Cref{lem:k-ar-rr-general,lem:reduction}, which were
used to prove \Cref{lem:privacy-free} in \Cref{sec:hypothesis_testing}.
\Cref{lem:k-ar-rr-general} is proved in \Cref{app:effect_of_randomized_response_on_light_elements} and \Cref{lem:reduction} is proved in \Cref{app:proof_of_cref_lem_reduction}.
\subsection{Proof of \Cref{lem:k-ar-rr-general}}
\label{app:effect_of_randomized_response_on_light_elements}
Recall the definitions of $A$ and $A'$ from equation~\Cref{eq:Adefinition}.
\LemkRRGeneral*
\begin{proof}
	Without loss of generality, we will assume that $ \sum_{i \in A
		} (\sqrt{q_i} - \sqrt{p_i})^2 \geq \frac{\tau}{2}$.
	Let $p' = \RR^{\epsilon,\ell } p$ and $q' = \RR^{\epsilon,\ell } q$.
	By the definition of the randomized response, each probability $x$ is mapped
	to $(1 + x(e^{\epsilon} - 1))/(k-1 + e^{\epsilon})$.
	Thus, $p'$ and $q'$ are given by
	\begin{align}
		\label{eq:q_and_q_after_transformation}
	p'_i & =
		\frac{1 + p_i(e^{\epsilon} - 1)}{(\ell -1) + e^\epsilon}, \quad \text{and} \quad
		q'_i =
		\frac{1 + q_i(e^{\epsilon} - 1)}{(\ell -1) +
			e^\epsilon}, \quad \forall i \in \ell.
	\end{align}
	Recall that $\delta_i = (p_i - q_i)/q_i \in [0,1]$. For each $i \in \ell$, we now define $\delta'_i := (p_i' - q_i')/q_i'$, which
	has the following
	expression in terms of $\delta_i$ and $q_i$:
	\begin{align}
		\label{eq:defn-after-transformation} \delta'_i & = \frac{p_i' - q_i'}{q_i'}
		=
		\frac{(e^{\epsilon} - 1) (p_i - q_i)}{1 + q_i(e^\epsilon-1)} =
		\frac{(e^{\epsilon} - 1) q_i}{1 + q_i(e^\epsilon-1)} \cdot \delta_i.
	\end{align}
	Let $r = 0.01 \min \left(e^{-\epsilon}, \frac{ \tau }{\ell }\right)$.
	We define the following subsets of the domain:
	\begin{align}
		\cE & = \{i:
		\delta_i \in (0,1] \textnormal{ and } q_i \geq e^{-\epsilon}\}\,, \\ \cE' & =
		\{i: \delta_i \in (0,1] \textnormal{ and } q_i \in (r, e^{-\epsilon})\}\,.
	\end{align}
	Observe that $\cE \cup \cE' \subseteq A$.

	Since $e^\epsilon \geq \ell $, equation~\Cref{eq:q_and_q_after_transformation} implies
	that
	$q'_i \geq \frac{1}{4}(e^{-\epsilon} + q_i)$.
	In particular, on $i \in \cE'$, we have $q'_i \geq 0.25 e^{-\epsilon}$, and on $i \in
		\cE$, we have $q_i' \geq 0.25 q_i$.

	We now apply these approximations to equation~\Cref{eq:defn-after-transformation}: we
	lower-bound the numerator by $0.5 e^{\epsilon} q_i \delta_i$ and upper-bound
	the denominator based on whether $i \in \cE$ or $i \in \cE'$.
	On $\cE'$, the denominator in equation~\Cref{eq:defn-after-transformation} is
	upper-bounded by $2$, and on $\cE$, the denominator is upper-bounded by
	$2q_ie^\epsilon$.
	This is summarized as follows: for $i \in \cE \cup \cE'$, we have
	\begin{align*}
		\delta'_i &
		\geq
		\begin{cases}
			0.1 \delta_i q_i e^\epsilon , & i \in \cE' \\ 0.1 \delta_i,
			& i \in \cE,
		\end{cases}
		\qquad \qquad q'_i \geq
		\begin{cases}
			0.25 e^{-\epsilon}
			, & i \in \cE' \\ 
			0.25 q_i, & i \in \cE
		\end{cases}.
	\end{align*}
	By definition
	of $\delta'$, it follows that $\delta_i' \in (0, 1]$ on $i \in \cE \cup \cE'$.
	Thus, the contribution from the $i^{\text{th}}$ element to $\hel^2(p',q')$ is at least a
	constant times $q'_i (\delta_i')^2$; see \Cref{cl:approx-sqRt}.
	Applying this element-wise, we obtain the following:
	\begin{align*}
		\hel^2(p',q') & \gtrsim \sum_{i \in \cE'}  q_i' (\delta'_i)^2  + \sum_{i \in \cE} q_i' (\delta'_i)^2
		\\ 
		& \gtrsim \sum_{i \in
			\cE'} e^{-\epsilon}  \left(0.1 \delta_i q_i e^{\epsilon}\right)^2 + \sum_{i
			\in \cE} q_i  \left(0.1\delta_i\right)^2
		 \\ 
&		\gtrsim e^\epsilon r \sum_{i \in \cE'} q_i \delta_i^2 + \sum_{i \in \cE} q_i
		\delta_i^2 \numberthis \label{eq:hel-decomposition}.
	\end{align*}
	Now consider the set $\cA = \{i: i \in A \textnormal{ and } q_i \geq r
		\}$, which is equal to $\cE \cup \cE'$.
	The set $\cA$ preserves the contribution to Hellinger divergence from
	comparable elements, as shown below:
	\begin{align*}
		\sum_{i \in \cA} (\sqrt{q_i} - \sqrt{p_i})^2 & =
		\sum_{i \in A} (\sqrt{q_i} - \sqrt{p_i})^2 -
		\sum_{i: i \in A, q_i \leq r} (\sqrt{q_i} - \sqrt{p_i})^2 
		 \geq \frac{\tau}{2} - 2\ell r 
		 \geq \frac{\tau}{4},
	\end{align*}
	since $r \leq \frac{\tau}{10 \ell }.
	$

	Since $\cA = \cE_1 \cup \cE_2$, one of the two terms $ \sum_{i \in
			\cE'} (\sqrt{q_i} - \sqrt{p_i})^2$ or $ \sum_{i \in \cE} (\sqrt{q_i} -
		\sqrt{p_i})^2$ must be at least $\frac{\tau}{8}$.

	Now consider the following two cases:
	\paragraph{Case 1: $ \sum_{i \in
				\cE} (\sqrt{q_i} - \sqrt{p_i})^2 \gtrsim \tau$.
	}
	In this case, we are done by inequality~\cref{eq:hel-decomposition}.
	That is,
	\begin{equation*}
	\hel^2(p',q') \gtrsim \sum_{i \in
			\cE} (\sqrt{q'_i} - \sqrt{p'_i})^2 \gtrsim \sum_{i \in
			\cE} q_i \delta_i^2 \gtrsim \sum_{i \in
			\cE} (\sqrt{q_i} - \sqrt{p_i})^2 \gtrsim \tau,
	\end{equation*}
	where we use
	\Cref{cl:approx-sqRt} element-wise.

	\paragraph{Case 2: $ \sum_{i \in
				\cE'} (\sqrt{q_i} - \sqrt{p_i})^2 \gtrsim \tau$.}

	By inequality~\cref{eq:hel-decomposition}, we have
	\begin{align*}
		\hel^2(p',q') \gtrsim
		e^\epsilon \cdot r  \sum_{i \in
			\cE'} q_i \delta_i^2 \gtrsim e^\epsilon \cdot r  \tau   \gtrsim \min\left( 1, e^\epsilon \frac{\tau}{\ell} \right) \tau,
	\end{align*}
	where we use the definition of $r$.

	Thus, we obtain the desired lower bound in both of the cases.
\end{proof}

\subsection{Proof of \Cref{lem:reduction}}
\label{app:proof_of_cref_lem_reduction}

\PropReduction*
\begin{proof}
	Let us begin by considering the case when $\sum_{i \in A \bigcup A'}
		\left(\sqrt{q_i} - \sqrt{p_i}\right)^2 \leq \frac{\hel^2(p,q)}{2}$.
	Following Pensia, Jog, and Loh~\cite[Theorem 3.2 (Case 1 in the proof)]{PenJL22}, there exists a
	binary channel that preserves the Hellinger divergence up to constants.
	This completes the case for $\ell=2$ above.

	Suppose for now that $\sum_{i \in A \bigcup A'} \left(\sqrt{q_i} -
		\sqrt{p_i}\right)^2 \geq \frac{\hel^2(p,q)}{2}$, i.e., the comparable elements
	constitute at least half the Hellinger divergence.
	Consider the channel $\bT'$ that maps the comparable elements of $p$ and
	$q$ to distinct elements, and maps the remaining elements to a single super-element.
	Let $\alpha$ be the contribution to the Hellinger divergence from the comparable
	elements in $\bT 'p$ and $\bT' q$ (defined analogously to
	equation~\Cref{eq:Adefinition}).
	It can be seen that $\alpha \geq \frac{\hel^2(p,q)}{2}$.
	Let $\ell \geq 3$ be as in the statement.
	Now consider the channel $\bT''$ that compresses $\bT'p$ and $\bT' q$ into
	$\ell$-ary distributions that preserve the Hellinger divergence, from Pensia, Jog, and Loh~\cite[Theorem 3.2 (Case 2 in the proof)]{PenJL22}.
	Let $\beta_\ell$ be the contribution to the Hellinger divergence from the
	comparable elements in $\bT''\bT 'p$ and $\bT''\bT' q$.
	Then the result in Pensia, Jog, and Loh~\cite[Theorem 3.2]{PenJL22} implies that $\beta_l \gtrsim
		\alpha \left( \ell /\min(k, 1 + \log(1/\hel^2(p,q))) \right)$.
	This completes the proof in this setting.
\end{proof}

\section{Properties of Private Channels} %
\label{sec:properties_of_private_channels}

Recall the definition of the set of channels $\cJ_{\ell,k}^{\gamma, \nu}$ from
\Cref{def:family-of-channels} below:
\DefLPFamily*
We begin by an equivalent characterization of the constraints above.
For a channel $\bT$ from $[k]$ to $[\ell]$, let $\{m_1, \dots, m_ \ell\}$ and
$\{M_1 , \dots, M_\ell\}$ be the minimum and maximum entries of each row,
respectively.
Then the channel $\bT$ satisfies the conditions~\Cref{eq:lp-convex-set} if
and only if for each $j \in [\ell] $, we have
\begin{align}
\label{eq:lp-set-equivalent}
M_j \leq \gamma_j m_j + \nu_j.
\end{align}

We first show that $\cJ_{\ell,k}^{\gamma, \nu}$ satisfies
\Cref{def:one_loose_entry}.
For the special case of LDP channels, the following claim was also proved in
Holohan, Leith, and Mason~\cite{HolLM17}:
\begin{claim}
	\label{claim:pvt_channels_one_loose_entry}
	$\cJ_{\ell,k}^{\gamma, \nu}$ satisfies \Cref{def:one_loose_entry}.
\end{claim}

\begin{proof}
	Let $\bT$ be any extreme point of $\cJ_{\ell,k}^{\gamma, \nu}$.
	Let $\{m_1,\dots, m_\ell\}$ and $\{M_1,\dots, M_\ell\}$ be as defined above.
	Suppose that there exists $c \in [k]$,
	such that there exist distinct $r, r' \in [\ell]$
	with $\bT(r,c) \in (m_r, M_r)$ and $\bT(r',c) \in (m_{r'}, M_{r'})$.
	In particular, both $\bT(r,c)$ and $\bT(r',c)$ are strictly positive and less
	than $1$.

	We will now show that $\bT$ is not an extreme point of $\cJ_{\ell,k}^{\gamma,
			\nu}$.
	For an $\epsilon > 0$ to be decided later, consider
	the channel $\bT'$ that is equal to $\bT$ on all but two entries:
	\begin{itemize}
		\item On $(r,c)$, $\bT'$ assigns probability
		      $\bT(r,c) + \epsilon$.
		\item On $(r',c)$, $\bT'$ assigns probability
		      $\bT(r',c) - \epsilon$.
	\end{itemize}
	Now define $\bT''$ similarly, with the difference being that on $(r,c)$, $\bT''$
	assigns probability
	$\bT(r,c) - \epsilon$, and on $(r',c)$, $\bT''$ assigns probability
	$\bT(r',c) + \epsilon$.
	Both $\bT'$ and $\bT''$ are thus valid channels for $\epsilon$ small enough.
	Let us show that $\bT'$ and $\bT''$ belong to $\cC$.
	If we choose $\epsilon > 0$ small enough,
	the row-wise maximum and minimum entries of $\bT'$ and $\bT''$
	are equal to those of $\bT$.
	Here, we critically use the fact that the entries that were modified were ``free.''
	By inequality~\Cref{eq:lp-set-equivalent}, both $\bT'$ and $\bT''$ belong to
	$\cJ_{\ell,k}^{\gamma, \nu}$.
	Since $\bT$ is the average of $\bT'$ and $\bT''$, it is not an extreme point of
	$\cJ_{\ell,k}^{\gamma, \nu}$.
\end{proof}

We now show that $\cJ_{\ell,k}^{\gamma, \nu}$ satisfies
\Cref{def:mass_transfer}.
\begin{claim}
\label{cl:LP-channels-condition}
	$\cJ_{\ell,k}^{\gamma, \nu}$ satisfies \Cref{def:mass_transfer}.
\end{claim}
\begin{proof}
	We follow the notation from \Cref{def:mass_transfer}.
	Let $\bT$ be an extreme point of $\cJ_{\ell,k}^{\gamma, \nu}$, and let $r$ and
	$r'$ be the corresponding rows.
	We show that $\bT'$ (defined in the condition) belongs to
	$\cJ_{\ell,k}^{\gamma, \nu}$ by showing that entries of $\bT'$ satisfy the
	constraints of the $r^{\text{th}}$ row and the $r'^{\text{th}}$ row (since the other rows are
	unchanged).
	In fact, we establish these arguments only for the $r^{\text{th}}$ row, and the analogous
	arguments hold for the $r'^{\text{th}}$ row.

	Let $m_r$ and $M_r$ be the row-wise minimum and maximum entry of this row in $\bT$.
	Let us first consider the case when $M_r< \gamma_r m_r + \nu_r$.
	Then there exist positive $\epsilon'$ and $\delta'$ such that
	$M_r + \delta < \gamma_r (m_r - \epsilon) + \nu_r$.
	By inequality~\Cref{eq:lp-set-equivalent},
	as long as the $r^{\text{th}}$ row of a channel contains entries in $[m_r -
				\epsilon, M_r + \delta]$, the constraints of this particular row will be
	satisfied.
	Since the entries in the $r^{\text{th}}$ row of $\bT'$ belong to this interval, the
	constraints of the $r^{\text{th}}$ row are satisfied by $\bT'$.

	Let us now consider the alternate case where $M_r = \gamma_r m_r + \nu_r$.
	Since $m$ and $M$ do not correspond to the min-tight and max-tight
	entries, we have $m_r < M$ and $m < M_r$.
	Consequently, even after perturbations by $\epsilon > 0$ and $\delta > 0$
	small enough, the entries of $\bT'$ lie in $[m_r, M_r]$.
	Thus, inequality~\Cref{eq:lp-set-equivalent} implies that the constraints of the $r^{\text{th}}$ row in
	$\bT'$ are satisfied.
\end{proof}

\begin{restatable}[Closure under pre-processing]{claim}{ClComposition}
	The set $\cJ_{\ell,k}^{\gamma, \nu}$ satisfies the following:
	\begin{align}
	\cJ_{\ell,k}^{\gamma, \nu} = \bigcup_{\ell'=1}^k \left\{\bT_2 \times \bT_1 : \bT_2 \in \cJ_{\ell, \ell'}^{\gamma, \nu} \textnormal{
			and } \bT_1 \in \cT_{\ell', k} \right\}.
	\end{align}
\end{restatable}
\begin{proof}
	We first show the simple direction that $\cJ_{\ell,k}^{\gamma, \nu} \subseteq
		\bigcup_{\ell'=1}^k \left\{\bT_2 \times \bT_1 : \bT_2 \in
		\cJ_{\ell,\ell'}^{\gamma, \nu} \textnormal{and } \bT_1 \in \cT_{\ell', k}
		\right\}$.
	Let $\bI_k$ correspond to the identity channel on $[k]$.
	Then every channel $\bT \in \cJ_{\ell,k}^{\gamma, \nu}$, can be written as
	$\bT \times \bI$.
	Thus, $\cJ_{\ell,k}^{\gamma, \nu} \subseteq \left\{\bT_2 \times \bI_k : \bT_2
		\in \cJ_{\ell,\ell'}^{\gamma, \nu}\right\}$, and the desired conclusion
	follows.

	We now show that every channel in the right-hand side belongs to
	$\cJ_{\ell,k}^{\gamma, \nu}$.
	For an arbitrary $\ell' \in [k]$, let $\bT_2 \in \cJ_{\ell,\ell'}^{\gamma,
			\nu}$.
	Define $\{m_1, \dots, m_ \ell\}$ and
	$\{M_1 , \dots, M_\ell\}$ to be the minimum and maximum entries of each row in
	$\bT_2$, respectively.
	By inequality~\Cref{eq:lp-set-equivalent}, for each $j \in [\ell]$, we have $M_j
		\leq \gamma_j m_j + \nu_j$.
	Let $\bT_1 \in \cT_{\ell',k} $ be an arbitrary channel.

	Let $\bT = \bT_2 \times \bT_1$ be in $\cT_{\ell,k}$,
	and let $\{m_1, \dots, m_ \ell\}$ and $\{M'_1 , \dots, M'_\ell\}$ be the minimum and
	maximum entries of each row in $\bT$, respectively.
	In order to show that $\bT \in \cJ_{\ell,k}^{\gamma, \nu}$, we need to show
	that for each $j \in [\ell]$, we have $M_j' \leq \gamma_j m_j' + \nu_j$.
	Since it already holds that $M_j \leq \gamma_j m_j + \nu_j$ for all $j$, it
	suffices to show that $M_j' \leq M_j$ and $m_j ' \geq m_j$ for all $j$.
	Observe that for any $c \in [k]$ and $r \in [\ell]$,
	the $(r,c)$-entry of $\bT$ is a convex combination of the $r^{\text{th}}$ row in
	$\bT_2$, where the weights in the convex combination correspond to the $c^{\text{th}}$
	column in $\bT_1$.
	Since the maximum of a collection of items is always as large as any convex
	combination of these items,
	we have $M_j' \leq M_j$ for all $j$.
	Similarly, we have $m_j' \geq m_j$.
	This completes the proof.
\end{proof}

\section{Other Notions of Privacy} %
\label{app:other_notions_of_privacy}
We provide the proof of the following result, omitted from
\Cref{sec:extensions_to_other_privacy}:
\LemSampCompApproxLDP*
\begin{proof}
	Let $\bT$ be an $\epsilon$-LDP channel that maximizes $\hel^2(\bT p, \bT
		q)$ among all $\epsilon$-LDP channels.
	Let $\bT'$ be the following channel that maps from $[k]$ to $[2k]$:
	for any element $i \in [k]$, use the channel $\bT$, and with probability
	$\delta$, map $i$ to $k + i$.
 It can be seen that $\bT'$ satisfies $(\epsilon,\delta)$-LDP.
	Let $p'$ and $q'$ be the corresponding distributions after transforming $p$ and
	$q$ using $\bT'$.
	It can be seen that $p'$ is a distribution over $[2k]$ such that the first $k$
	elements are equal to $(1 - \delta)\bT p$ coordinate-wise, and the bottom $k$
	elements are equal to $ \delta p$ coordinate-wise.
	A similar conclusion holds for $q'$, as well.
	Thus, we have
	\begin{align*}
\hel^2(\bT' p, \bT'q) &= (1 - \delta) \cdot \hel^2(\bT p, \bT q) + \delta \cdot \hel^2(p,q) \\
&\asymp \max\left(  \left( 1 - \delta \right) \cdot \hel^2(\bT p, \bT q), \delta \cdot \hel^2( p, q)  \right) \\
&\asymp \max\left(  \left( 1 - \delta \right) \cdot \frac{1}{\nstar(p,q,\epsilon)}, \delta \cdot \frac{1}{\nstar(p,q)}  \right).
\end{align*}
	By \Cref{fact:sample-complexity},
	the sample complexity $\nstar(p,q,(\epsilon,\delta))$ is at most $1/\hel^2(\bT
		'p, \bT 'q)$, which gives the upper bound on $\nstar(p,q,(\epsilon,\delta))$.

	The tightness follows from the result of Kairouz, Oh, and Viswanath~\cite[Theorem 18]{KaiOV16}, which
	implies that $\bT'$ defined above is an optimal channel for binary distributions.
\end{proof}

\section{Auxiliary Lemmas} %
\label{sec:supplementary}

\subsection{Degenerate Conditions for Joint Range}
\label{sec:regu-cond-joint-range}
We show in this section that we can safely rule out certain degenerate
conditions for $p$ and $q$
for our results.
Let $p$ and $q$ be two distributions on $[k]$.
In particular, we would like to assume the following:
\begin{itemize}
	\item Consider the likelihood ratio $p_i/q_i$, defined to be $\infty$
	      if $q_i = 0$ and $p_i \ne 0$, and undefined if both $p_i$ and $q_i$ are 0.
	     Assume that all the likelihood ratios are well-defined and unique.
\end{itemize}
If these conditions do not hold, define $p'$ and $q'$ to be distributions over
$[k']$ for some $k' \le k$, constructed as follows: start by removing elements
that have zero probability mass under both $p$ and $q$, then merge the elements
with the same likelihood ratios into super-elements.
Let $\bT^* \in \cT_{k',k}$ be the corresponding deterministic map, which
satisfies $p' = \bT^* p$ and $q' = \bT^* q$.
We make the following claim:
\begin{claim}
	\label{claim:degenerate-p-q}
	With the notation above, for any $\ell \in \N$ and $\bT \in \cT_{\ell,k}$,
	there exists $\bT' \in \cT_{\ell,k'}$ such that
	$(\bT p, \bT q) = (\bT' p', \bT' q')$.
	In particular, $\{(\bT p, \bT q): \bT \in \cC \} = \{(\bT p', \bT q'): \bT \in
		\cC' \}$ for two choices of $\cC$ and $\cC'$: (i) $(\cC, \cC') = (\cT_{\ell,k},
		\cT_{\ell,k'})$ and (ii) $(\cC, \cC') = (\cP^\epsilon_{\ell,k},
		\cP^\epsilon_{\ell,k'})$.
\end{claim}
\Cref{claim:degenerate-p-q} ensures that the joint ranges of $(p,q)$ and $(p',q')$ are identical, so our structural and algorithmic results continue to hold when applied to $p'$ and $q'$.
We will now prove \Cref{claim:degenerate-p-q}.
\begin{proof}[Proof of \Cref{claim:degenerate-p-q}]
	Let $\{\cI_0, \cI_1, \dots, \cI_{k'}\}$ be the smallest partition of $[k]$ such that $\cI_0$ contains elements where both $p_i$ and $q_i$ are zero, and
	for each $i \in [k']$, the likelihood ratio of elements in $\cI_i$ are identical.
	Then the channel $\bT^*$ mentioned above has the following form: $\bT^*(x) = i$
	if $x \in \cI_i$ and $i > 0$, and $\bT^*(x) = 1$ if $x \in \cI_0$.
	Observe that for each $i \in [k']$, we have
	$p'_i = \sum_{j \in \cI_i} p_j$,
	$q'_i = \sum_{j \in \cI_i} q_j$,
	and at most one of them is zero.

	Now consider a channel $\bT \in \cT_{\ell,k}$, and let $\{v_1, \dots, v_k\}$ be the
	columns of $\bT$.
	It is easy to see that columns belonging to indices in $\cI_0$ do not affect
	$(\bT p, \bT q)$.
	For $i \in [k']$, define $\theta'_i := p'_i/q'_i$ to be the likelihood ratio of
	the transformed distributions.
	Define $\bT'$ to be the channel with columns $v'_1, \dots, v'_k$ such that
	\begin{align*}
v'_i = \begin{cases}\frac{\sum_{j \in \cI_i} v_j p_j}{p'_i} & \text{ if } p'_i > 0, \\
\frac{\sum_{j \in \cI_i} v_j q_j}{q'_i} & \text{ otherwise}.\end{cases}  
\end{align*}

	First consider the case when for all $i \in [k']$, we have $0 < \theta'_i < \infty$.
	Then for all $i \in [k']$, we have $p'_i = \theta'_i q'_i$
	and $v'_i = \frac{\sum_{j \in \cI_i} v_j p_j}{p'_i} = \frac{\sum_{j \in \cI_i}
			v_j q_j}{q'_i}$.
	Thus, we have
	\begin{align*}
(\bT p, \bT q) 
	&= \left(\sum_{i \in [k']}\sum_{j \in \cI_i} v_j p_j, \sum_{i \in [k']} \sum_{j \in \cI_i}v_j q_j\right)\\
	&= \left(\sum_{i \in [k']} p'_i \cdot \left( \frac{\sum_{j \in \cI_i} v_j p_j}{p'_i}  \right) , \sum_{i \in [k']} q'_i \cdot \left( \frac{\sum_{j \in \cI_i} v_j q_j}{q'_i}\right)\right)\\
	&= \left(\sum_{i \in [k']} p'_i v'_i, \sum_{i \in [k']} q'_i v'_i\right) = \left( \bT'p', \bT' q' \right).
\end{align*}
	We now consider the case when there is an index $a \in [k']$ such that $p'_{a}
		= 0$ and an index $b \in [k']$ such that $q'_{b} = 0$.
	Then it must be that $\theta'_{a} = 0$ and $\theta'_b = \infty$.
	Then $v'_{a} = \frac{\sum_{j \in \cI_i} v_j q_j}{q'_i}$ and $v'_b =
		\frac{\sum_{j \in \cI_i} v_j p_j}{p'_i}$.
	Following the calculations above, we obtain $\sum_{j \in \cI_i} v_j p_j = v'_i p'_i$ for each $i \in [k]\setminus
		\{a\}$.
	In fact, the same result is true for $i = a$, since both sides are 0.
	The same conclusion holds for $q$ and $q'$, as well.
	This completes the proof of the first claim.

	We now turn to the final claim, regarding the joint range under the channel
	constraints of $\cC$.
	The case $\cC = \cT_{\ell,k}$ is immediate from the preceding discussion.
	Let $\bT_1 \in \cT_{k,k'}$ be such that $(p,q) = (\bT_1 p' ,\bT_1 q')$ and $\bT_2
		\in \cT_{k',k}$ be such that $(p',q') = (\bT_2 p ,\bT_2 q)$.
	For $\cC = \cP^\epsilon_{\ell,k}$ and $\cC' = \cP^\epsilon_{\ell,k'}$,
	we only need to show that
	(i) if $\bT' \in \cC'$, then $\bT' \bT_2 \in \cC$; and (ii) if $\bT \in \cC$,
	then $\bT\bT_1 \in \cC'$.
	Both of these conditions hold because privacy is closed under pre-processing.
\end{proof}

\subsection{Valid Choice of Parameters in \Cref{lem:worst-case-samp-comp}}
\label{sec:valid-choice-worst-case}

    We now give the details that were omitted in the proof of \Cref{lem:worst-case-samp-comp} in \Cref{sec:worst-case-samp-comp-lwr-bds}.

    We first reparametrize the problem by setting $x = \gamma$ and $y = \gamma^{1+\delta}$. 
    The constraint $\delta > 0$ is equivalent to $y < x$.
    Then $\dtv(p,q) = x + y$, and
        $$
        \hel^2(p,q) = 2 y + \left(\sqrt{1/2 + x - y} - \sqrt{1/2} \right)^2 + \left(\sqrt{1/2 - x - y} - \sqrt{1/2} \right)^2 .
        $$
    We begin by setting $\nu = x + y$, which is possible since $0 \leq y < x < 0.25$ and $\nu \in (0, 0.5)$.
    Then $x = \nu - y$, where $y \in (0, \nu/2)$ and $\nu \in (0,0.5)$.   
    Our goal is now to show that there exists a valid choice of $y$ such that $\hel^2(p,q) = \rho$, as long as $2\nu^2 \leq \rho \leq \nu$.

    Define $g(y)$ to be the Hellinger divergence between $p$ and $q$ given $y$, i.e.,
    $$ g(y) = 2 y + \left(\sqrt{1/2 + \nu - 2 y} - \sqrt{1/2} \right)^2 + \left(\sqrt{1/2 - \nu} - \sqrt{1/2} \right)^2.
    $$
    Since $g$ is a continuous function, it suffices to show that $g(0) < 2\nu^2$ and $g(\nu/2) > \nu$, which would imply that there is a choice of $y \in (0, \nu/2)$ such that $g(y) = \rho$.
    We have 
    $$g(0) = \left(\sqrt{1/2 + \nu} - \sqrt{1/2} \right)^2 + \left(\sqrt{1/2 - \nu} - \sqrt{1/2} \right)^2 \leq 3 \nu^2/2, $$
    where we use the fact that $\left|\sqrt{1/2 + a }-\sqrt{1/2}\right| \leq a$ for all $a \geq 0$, and is less than $|a|/2$ for $a \leq 0$.
    On the other hand, $g(\nu/2) > \nu$, since $\nu < 1/2$.
    Thus, there is a choice of $y \in (0, \nu/2)$ such that $\hel^2(p,q) = \rho$.
    Given these choices of $x$ and $y$, we can infer the choice of $\gamma \in (0,0.25)$ and $\delta > 0$.

\subsection{Taylor Approximation to Hellinger Divergence}

\ClApproxHel*
\begin{proof}
	It suffices to prove that for $\delta \in (0,1]$, we have $1 - \sqrt{1
			-\delta} \asymp \delta$.
	We first start with the upper bound: since $1 - \delta \leq \sqrt{1 - \delta}$,
	we have 
	$1 - \sqrt{1 - \delta} \leq \delta$.
	We now show the lower bound and claim that $1 - \sqrt{1 - \delta} \geq 0.5
		\delta$ for all $\delta \in [0,1]$.
	This inequality is equivalent to showing $1 - 0.5 \delta \geq \sqrt{1 -
			\delta}$, which is equivalent to showing that $1 + 0.25 \delta^2 - \delta \geq 1 -
		\delta$, which holds since $\delta^2 \geq 0$.
\end{proof}

\ApprxHel*
\begin{proof}
	Let $q$ be the larger of the two quantities, so $p$ satisfies $p \leq
		\frac{1}{2}$.
	The total variation distance is thus $q-p$.
	Let $\delta = (q-p)/q \in (0,1]$.
	Observe that $p = q - q\delta$ and the total variation distance is $\delta q$.

	We begin by noting that \Cref{cl:approx-sqRt} implies that
	\begin{align}
	\label{eq:bin-hell-part1}
	\left(\sqrt{q} -
		\sqrt{p}\right)^2
        = \left(\sqrt{q} -
		\sqrt{q - \delta q}\right)^2 \asymp \frac{\delta^2 q^2}{q} \asymp \frac{\dtv^2\left(\Ber(p), \Ber(q)\right)}{q}.
	\end{align}

	We now split the analysis into two cases:

	\paragraph{Case 1: $q \leq 1/2$.} %
	\label{par:case_1_}
	Then Pensia, Jog, and Loh~\cite[Claim E.2]{PenJL22} implies that $\hel^2( \Ber(p), \Ber(q))
		\asymp(\sqrt{q} - \sqrt{p})^2$.
	Thus, equation~\Cref{eq:bin-hell-part1} implies the result.

	\paragraph{Case 2: $q \geq 1/2$.} %
	\label{par:case_2_}

	Applying \Cref{cl:approx-sqRt} again to the second term, we obtain
	\begin{align}
	\label{eq:bin-hell-part2}
	\left(\sqrt{1-p} - \sqrt{1- q}\right)^2
		= \left( \sqrt{1 - p} - \sqrt{1 - p - q \delta} \right)^2
		\asymp \frac{q^2 \delta^2}{1-p} 
		\asymp \frac{q^2 \delta^2}{q}   \asymp \frac{\dtv^2(\Ber(p), \Ber(q))}{q},
		\end{align}
	where we use the fact that $1- p \asymp q$, since $p,q \in [0.5,1]$.
	The desired conclusion follows from
	equations~\Cref{eq:bin-hell-part1} and~\Cref{eq:bin-hell-part2}.
	\end{proof}

\end{document}